\documentclass[12pt, oneside,final, letter]{ucr}
\usepackage{anysize}
\usepackage[left=1.5in,top=1.5in,right=1in,bottom=1in]{geometry}
\usepackage{xspace}
\usepackage{graphicx}
\usepackage{pdfsync}
\usepackage{amsmath}
\usepackage{amssymb}
\usepackage{bm}
\usepackage{amsthm}
\usepackage{url}
\usepackage{fncylab}
\usepackage[usenames]{color}
\usepackage{tikz}
\usetikzlibrary{matrix,arrows}
\usepackage{slashed}

\newtheorem{lemma}{Lemma}[section]
\newtheorem{proposition}{Proposition}[section]
\newtheorem{theorem}{Theorem}[section]
\theoremstyle{definition}
\newtheorem{definition}{Definition}[section]
\newtheorem{remark}{Remark}[section]

\newtheorem{corollary}{Corollary} [section]
\def\dsp{\def\baselinestretch{2.0}\large\normalsize}
\dsp

\begin{document}

% Declarations for Front Matter

\title{Seiberg-Witten Invariants, Alexander Polynomials, and Fibred Classes}
\author{Oliver James Thistlethwaite}
\degreemonth{December}
\degreeyear{2014}
\degree{Doctor of Philosophy}
\chair{Dr. Stefano Vidussi }
\othermembers{Dr. Yatsun Poon\\
Dr. Bun Wong}
\numberofmembers{3}
\field{Mathematics}
\campus{Riverside}

\maketitle
\copyrightpage{}
\approvalpage{}

\degreesemester{Fall}

\begin{frontmatter}

\begin{acknowledgements}
I thank my advisor, Stefano Vidussi, for all his help and his endless patience with me.

I am grateful to my mother, Annette Thistlethwaite, for all her love and support and my father, Morwen Thistlethwaite, for inspiring me to become a mathematician. I also thank my siblings, Abigail Thistlethwaite and William Thistlethwaite for their encouragement.

I was fortunate to have made many great friends at UCR. I am grateful to all the fellow graduate students I have met throughout my years here. In somewhat of a chronological order, I would particularly like to thank the following people. Jorge Perez, for being a great friend and our many conversations at Farmer Boys. I thank Chunghoon Kim for his friendship and all his help. Reid Livingston, for always being willing to grab a pitcher at the local Getaway cafe. Edward Burkard, for his friendship and hospitality in letting me stay at his apartment. Jason Chou, for the countless hours we spent talking outside the math department building. I am grateful to Jason Park for always being willing to listen to my complaints. Mat Lunde, for cooking me delicious food countless times. Jake West, for always providing valuable advice. Soheil Safii, for being a great guy and also the only other topology graduate student for many years. Dominick Scaletta, for showing me it's possible for someone to turn their life around and become an excellent mathematician. Teo Holbert, for motivating me with his enthusiasm for mathematics. I thank Jeff Wand for his cheerfulness. Peri Shereen, for being a great officemate and inspiring me with her hard work and positive attitude. Thomas Schellhous, for organizing many enjoyable events. Jason Erbele, for our many interesting conversations about math. I am grateful to Lisa Schneider for her kindness and the times we spent together. I thank Matt O'dell for always being willing to discuss obscure facts about Lie algebras and Professor Vyjayanthi Chari for making me feel welcome as an honorary member in the UCR Lie algebra group.
\end{acknowledgements}

\begin{dedication}
\null\vfil
{
\begin{center}
To my mother, Annette Thistlethwaite.
\end{center}}
\vfil\null
\end{dedication}

\begin{abstract}

Since their introduction in 1994, the Seiberg-Witten invariants have become one of the main tools used in $4$-manifold theory. In this thesis, we will use these invariants to identify sufficient conditions for a $3$-manifold to fibre over a circle. Additionally, we will construct several examples of genus $1$ and $2$ surface bundles and prove their total spaces are spin $4$-manifolds.
\end{abstract}

\tableofcontents
% \listoffigures
% \listoftables
\end{frontmatter}

% \part{First Part}

 \chapter{Introduction}

The past three decades have experienced a massive growth in four dimensional manifold theory. Very little was known about these manifolds until the early eighties. Smooth $4$-manifolds are very different from their higher and lower dimensional counterparts. As an example, four dimensional Euclidean space $\mathbb R^4$ is the only Euclidean space to admit more than one distinct smooth structure. In fact, it was shown in 1987 by Clifford Taubes that there are an uncountable number of so-called exotic smooth structures.

Relating to this dissertation, the most important development is that of the Seiberg-Witten invariants in 1994 by the physicists Nathan Seiberg and Edward Witten. Given a smooth closed oriented $4$-manifold $X$, these invariants can be used to define a map
$$SW : Spin^c(X) \rightarrow \mathbb Z$$
where $Spin^c(X)$ is the set of equivalence classes of $Spin^c$-structures on $X$.

Many exciting results can be obtained from these invariants. Our main interest will be how the invariants apply to the study of symplectic manifolds. In \cite{taubes}, Taubes showed a symplectic manifold's canonical $Spin^c$-structure is sent to $1$. This has the immediate consequences that if a symplectic $4$-manifold splits as a smooth
connected sum then one of the summands must have a negative definite intersection form and also a symplectic $4$-manifold admitting a metric with positive scalar curvature must have an almost-definite intersection form.

Fibre bundles will play a significant role in this dissertation. One of the purposes of this dissertation is use the relationship between the Seiberg-Witten invariants and $3$-manifold invariants given a $4$-manifold that fibres over a $3$-manifold.

$3$-manifolds have their own Seiberg-Witten invariants as well as twisted Alexander polynomials. Alexander polynomials originated as invariants for knots in $S^3$ and were defined by J. Alexander in 1923. In 1990, X. S. Lin introduced a generalization of these polynomials, called twisted Alexander polynomials. This definition was later generalized to $3$-manifolds. Meng and Taubes \cite{meng} showed these polynomials have a surprising connection to a $3$-manifold's Seiberg-Witten invariants.

In turn, given a principal circle bundle over a $3$-manifold, $X \rightarrow Y$, with nontorsion Euler class and $b_2(Y) = 2$, Baldridge \cite{bald2} found a relationship between the two manifolds' Seiberg-Witten invariants.

This dissertation will be organized as follows. Chapter 2 will be dedicated to providing some of the preliminary results necessary to understand the Seiberg-Witten invariants. These include some basic facts about symplectic structures which are closely related to the invariants. We will also prove some standard results about Clifford algebras as well as define $Spin$ and $Spin^c$ structures.

In Chapter 3, we will construct a detailed explanation of the invariants. We will need two maps, the Dirac operator and the curvature map. Using these maps, we can define a quotient space from which the Seiberg-Witten invariants arise. We will primary be interested in manifolds with $b^+ = 1$, however this slightly complicates the invariants so we dedicate a section to this case.

Chapter 4 will be dedicated to $3$-manifolds. We will define twisted Alexander polynomials and discuss their relationship to $3$-dimensional Seiberg-Witten invariants. Finally we will prove some results involving principal circle bundles.

In chapter 5, we prove a partial converse of Fern\'andez-Gray-Morgan's theorem. Fernandez-Gray-Morgan's theorem gives sufficient conditions for the total space of a principal circle bundle to admit a symplectic structure. These include the base space fibreing over a circle.  Our partial converse will give sufficient conditions for a principal circle bundle whose total space is symplectic to have a base space which fibres over a circle.

Chapter 6 will be dedicated to constructing some examples of surface bundles over tori and then proving their total spaces are in fact spin manifolds. This is to work towards answering a question posed by Ron Stern, ``is there an orientable aspherical surface bundle over the torus that is not spin'' \cite{hillman}.

 %usually intro
 \chapter{Background and Preliminaries}

In this chapter we will provide some of the required prerequisites to understand this dissertation. All these results are standard and we will provide references when necessary.
 
\begin{section}{Symplectic manifolds}

In this section, we give an introduction to symplectic manifolds as well as related concepts.

\begin{definition}
An {\it almost-complex structure} on a $2n$-manifold $X$ is a vector bundle map $J \colon TX \rightarrow TX$ satisfying
$J^2 = -id_{TX}$ where $TX$ denotes the tangent bundle of $X$. We call the pair $(X, J)$ an {\it almost-complex manifold}. 
\end{definition}

Note the set of almost-complex structures on $X$ is a subset of $C^\infty(X, TX \otimes T^*X)$ so it inherits a topology from the compact-open topology on $C^\infty(X, TX \otimes T^*X)$.

\begin{lemma} 
If $X$ be a $2n$-manifold, an almost-complex structure $J$ induces a complex vector bundle structure on $TX \xrightarrow[]{\pi} X$.
\end{lemma}
\begin{proof}
For each $v \in T_x X$ define $(a+ib)v = av + J(bv)$ for each complex number $a + ib$. So we have made each fibre of $TX$ into a complex vector space isomorphic to $\mathbb C^n$. Since $J$ is a fibrewise real-linear map, it follows each $x \in X$ possesses a neighborhood $U$ whose
inverse image $\pi^{-1} (U)$ is diffeomorphic to $U \times \mathbb C^n$ under a diffeomorphism that is fibrewise complex-linear. 
\end{proof}

So we see it makes sense to discuss Chern classes of almost-complex structures which we will denote $c_i (X, J) \in H^{2i}(X; \mathbb Z)$. Also note $J$ naturally defines an orientation on $X$ as it reduces the structure group of $TX$ to $GL(n,\mathbb C) \subset GL^+(2n)$.

\begin{definition}
The {\it canonical bundle} $K$ of an almost-complex manifold $(X, J)$ is defined to be the complex line bundle corresponding with $-c_1(X, J)$. 
\end{definition}

A two-form $\omega$ on a real $n$-dimensional vector space $V$ is called {\it nondegenerate} if for each nonzero vector $v \in V$ there exists $u \in V$ such that $\omega(v,u) \neq 0$. A $2$-form $\omega$ on $X$ is called {\it nondegenerate} if it is nondegenerate on each tangent space $T_x X$. $\omega$ is called {\it closed} if $d\omega = 0$.

\begin{definition}
A closed, nondegenerate $2$-form $\omega$ on $X$ is called a {\it symplectic form} on $X$. We call the pair $(X, \omega )$ a symplectic manifold.
\end{definition}

A two-form $\omega$ on a $4$-manifold $X$ is called {\it compatible with an almost-complex structure} $J$ if $<v_1, v_2> = \omega(v_1, Jv_2)$ defines a Riemannian metric on $X$. 
Similarly, a Riemannian metric $g$ is called {\it compatible an almost-complex structure} $J$ if $g(v_1, v_2) = g(Jv_1, Jv_2)$ for each $v_i \in \Gamma(TX)$.

In the case where a symplectic form $\omega$ and Riemannian metric $g$ are both compatible with an almost-complex structure $J$, the elements of the triple $(\omega, J, g)$ are called {\it compatible} with each other.

\begin{proposition} \cite{mcduff} 
Each symplectic manifold $(X, \omega)$ admits a compatible almost-complex structure $J$ and a compatible Riemannian metric $g$. The set of compatible almost-complex structures is contractible. \hfill $\Box$
\end{proposition}

Now suppose $(X, \omega )$ is a symplectic manifold. Let $J$ be a compatible almost-complex structure. Note we may define
the Chern classes $c_i (X, \omega, J) \in H^{2i}(X; \mathbb Z)$. As a result of the space of compatible almost-structures being contractible (hence connected), the $c_i (X, \omega, J)$ are independent of the choice of compatible almost-complex structure $J$ so we may define $c_i (X, \omega)$. 

\begin{definition}
The canonical class $K$ of $(X, \omega)$ is defined to be the cohomology class $-c_1 (X, \omega) \in H^2(X)$. 
\end{definition}

\begin{remark} We will denote the inverse of the canonical form $K$ as $-K$ and $K^{-1}$ interchangeably.
\end{remark}

\end{section}

\begin{section}{Clifford algebras}

In order to define Seiberg-Witten invariants, we will dedicate a section to Clifford algebras. This section will be entirely algebraic. All results that are not proven here may be found in \cite{SpinGeo}. 
Throughout this section, let $K$ denote the field $\mathbb R$ or $\mathbb C$.

\begin{definition} The {\it tensor algebra on $K^n$} is defined to be the algebra
$$T(K^n) =  \oplus_{r=0}^{\infty} T^r( K^n ) = K \oplus K^n \oplus K^n \otimes K^n \oplus \cdots$$
whose multiplication is given by linearly extending the canonical map
$$T^\ell (K^n) \times T^m (K^n) \rightarrow T^{\ell + m}(K^n).$$
\end{definition}

Note $T(K^n)$ is associative with multiplicative identity $1 \in K = T^0(K^n) \subset T(K^n)$. 

\begin{definition}
The {\it Clifford Algebra on $K^n$} is defined to be the quotient of algebras 
$$\mathcal C l(K^n) = T(K^n ) / I (K^n)$$
where $I(K^n )$ is the two-sided ideal generated by elements of form 
$$v \otimes v \; + < v, v >1 \in T(K^n) \,\,\, \text{ with } \,\,\,  v \in K^n.$$
Here $<\, ,>$ denotes the usual inner product on $K^n$, also known as the dot product.
\end{definition}

Note $\mathcal C l(K^n)$ has the multiplicative identity, $[1] \in [K] \subset \mathcal C l(K^n)$. We will denote $[1]$ with $1$.

\begin{lemma} \label{cliffembed} There is a natural embedding of $K^n = T^1(K^n)$ into $\mathcal C l(K^n)$.
\end{lemma}

\begin{proof}

It is sufficient to show $K^n \cap I(K^n ) = \{ 0 \}$. First an element in $T(K^n)$ is said to be of {\it pure degree} $s$ if it is contained in $T^s(K^n) \subset T(K^n)$. Let $\varphi \in K^n \cap I(K^n)$. Since $\varphi \in I(K^n)$ we may write it as a finite sum
$$ \varphi = \sum_i a_i \otimes (v_i \otimes v_i + <v_i, v_i> ) \otimes b_i$$
where the $a_i$ and $b_i$ are each of pure degree.
Since $\varphi \in K^n = T^1(K^n)$, we have 
$$ \sum_{i^\prime} a_{i^\prime} \otimes (v_{i^\prime} \otimes v_{i^\prime}) \otimes b_{i^\prime} = 0$$
where the above sum is taken over indices where each $\deg a_{i^\prime} + \deg b_{i^\prime}$ is maximal. 
By contraction with $<\, , >$, we also have
$$ \sum_{i^\prime} a_{i^\prime} <v_{i^\prime}, v_{i^\prime} > b_{i^\prime} = 0.$$
Hence
$$ \sum_{i^\prime} a_{i^\prime} \otimes (v_{i^\prime} \otimes v_{i^\prime} + <v_{i^\prime}, v_{i^\prime}> ) \otimes b_{i^\prime} = 0.$$
and proceeding inductively, we conclude $\varphi = 0$.
\end{proof}

\begin{lemma} \label{cliffbasis} Let $\{e_1, \cdots , e_n \}$ denote an orthonormal basis for $K^n$. Then $\mathcal C l(K^n)$ is generated as an algebra by the $e_i$ subject to the relations:
$$ e_i^2 = -1 \, \, \, \, \text{  and  } \, \, \, \, e_i e_j = -e_j e_i$$
for each $i$ and $j \neq i$.

It follows $\mathcal C l(K^n)$ is a $2^n$-dimensional vector space with basis
$$\{1 \} \cup \{e_{i_1} e_{i_2} \cdots e_{i_r} \mid 1 \leq i_1 < i_2 < \cdots i_r \leq n\}.$$
 \hfill $\square$
\end{lemma}

For the reader's convenience, we recall the definition of the complexification of a real algebra.

\begin{definition} If $A$ is an algebra over $\mathbb R$ then its complexification is the algebra $A \otimes_\mathbb R \mathbb C$ endowed with the following complex scalar multiplication map. For each $v \otimes z \in A \otimes_\mathbb R \mathbb C$ and $\lambda \in \mathbb C$, we define $\lambda (v \otimes z) =  v \otimes (\lambda z)$. 
\end{definition}

\begin{lemma} \label{complexcliff}
The complexification of $\mathcal C l (\mathbb R^n)$, $\mathcal C l (\mathbb R^n) \otimes \mathbb C$, is isomorphic as an algebra to $\mathcal C l (\mathbb C^n)$.
\end{lemma}

\begin{proof}

Define an algebra isomorphism $\mathcal C l (\mathbb R^n) \otimes \mathbb C \rightarrow \mathcal C l (\mathbb C^n)$ by 
$$v \otimes \lambda \mapsto \lambda v \in \mathbb C^n \subset C l (\mathbb C^n)$$
for each $v \in \mathbb R^n \subset \mathcal C l(\mathbb R^n)$ and $\lambda \in \mathbb C$.

\end{proof}

We will be interested in some low dimensional cases so we will specialize to those cases whenever useful.

\begin{lemma} \label{cliffmult}
We can construct an algebra isomorphism:
$$\mu : \mathcal Cl(\mathbb C^4) \rightarrow Mat(\mathbb C, 4).$$
\end{lemma}

\begin{proof}
Let $\{e_1, \cdots, e_4 \}$ denote the usual orthonormal basis for $\mathbb C^4 \subset \mathcal Cl(\mathbb C^4)$. We can define an algebra isomorphism $\mathcal Cl(\mathbb C^4) \rightarrow Mat(2, \mathbb C) \otimes Mat(2, \mathbb C)$ as follows.
$$ \begin{array}{l l}
      e_1 \mapsto   \left[ \begin{array}{cc} i & 0 \\ 0 & -i \end{array} \right] \otimes \left[ \begin{array}{cc} 0 & i \\ -i & 0 \end{array} \right] &
      e_2 \mapsto   \left[ \begin{array}{cc} 0 & 1 \\ -1 & 0 \end{array} \right] \otimes \left[ \begin{array}{cc} 0 & i \\ -i & 0 \end{array} \right] \\ & \\
      e_3 \mapsto   \left[ \begin{array}{cc} 1 & 0 \\ 0 & 1 \end{array} \right] \otimes \left[ \begin{array}{cc} 0 & i \\ i & 0 \end{array} \right] &
      e_4 \mapsto   \left[ \begin{array}{cc} 1 & 0 \\ 0 & 1 \end{array} \right] \otimes \left[ \begin{array}{cc} i & 0 \\ 0 & -i \end{array} \right]
   \end{array} 
$$
By composing this map with the Kronecker map, we obtain the desired algebra isomorphism, $\mu$.
\end{proof}

Using the map $\mu$ from the previous Lemma, we have an action of $\mathcal Cl(\mathbb C^4)$ on $\mathbb C^4$ which we will refer to as {\it Clifford multiplication }. We will denote this action by ``$\cdot$".

Now we will discuss the splittings of Clifford algebras. First note by Lemma~\ref{cliffbasis}, the map $\alpha : K^n \rightarrow K^n$ given by $\alpha(v) = -v$ extends to an algebra automorphism of $\mathcal C l(K^n)$. $\alpha$ induces the splitting
$$ \mathcal C l(K^n) = \mathcal C l_0 (K^n) \oplus \mathcal C l_1(K^n)$$
where $\mathcal C l_i (K^n) = \{ \varphi \in \mathcal C l(K^n) \mid \alpha(\varphi) = (-1)^i \varphi \}$. To verify this is indeed a splitting note each $C l_i (K^n)$ is a linear subspace, $span \{ C l_0 (K^n), C l_1 (K^n) \} = C l (K^n)$ as each element of the vector space basis $\{1 \} \cup \{e_{i_1} e_{i_2} \cdots e_{i_r} \mid 1 \leq i_1 < i_2 < \cdots i_r \leq n\}$ is contained in one of the $C l_i (K^n)$, and $C l_0 (K^n) \cap C l_1 (K^n) = 0$ since if $\varphi \in C l_0 (K^n) \cap C l_1 (K^n)$ then $\varphi = -\varphi$ implies $\varphi = 0$ as $K$ has no nontrivial torsion.

The {\it volume element} of $C l(\mathbb R^n)$ (oriented by the usual orientation of $\mathbb R^n$) is defined to be $\omega = e_1 \cdots e_n$ where the $e_i$ are an orthonormal basis of $\mathbb R^n$ with the usual orientation. To see this is well-defined, suppose $e_1^\prime, \cdots, e_n^\prime$ is another oriented orthonormal basis. Then each $e_i^\prime = \sum_j g_{ij} e_j$ for some $g = (g_{ij}) \in SO(n)$. Hence from Lemma~\ref{cliffbasis} 
$$e_1^\prime \cdots e_n^\prime = \det (g) e_1 \cdots e_n = e_1 \cdots e_n.$$

In the case $\omega^2 = 1$, we have the splitting
$$ \mathcal C l(\mathbb R^n) = \mathcal C l^+ (\mathbb R^n) \oplus \mathcal C l^-(\mathbb R^n)$$
where $\mathcal C l^\pm (\mathbb R^n) = \pi^\pm \mathcal C l (\mathbb R^n)$ with $\pi^\pm = \frac{1}{2}(1\pm \omega)$ (note we use this notation since left multiplication by $\pi^\pm$ is a projection onto $\mathcal C l^\pm (\mathbb R^n)$)  . To see this is a splitting, observe each $C l^\pm (\mathbb R^n)$ is a linear subspace, $span \{ C l^+ (\mathbb R^n), C l^- (\mathbb R^n) \} = C l (K^n)$ since we can write each $\varphi \in C l (\mathbb R^n)$ as $\varphi = \pi^+\varphi + \pi^-\varphi$, and $C l^+ (\mathbb R^n) \cap C l^- (\mathbb R^n) = 0$ since if $\pi^+ \varphi_1 = \pi^- \varphi_2$, using the facts  
$\pi^\pm \pi^\pm = \pi^\pm \,\,\,\, \text{and} \,\,\,\, \pi^\pm \pi^\mp = 0,$
we obtain $0 = \pi^- \varphi_2$. Note for each $e \in \mathbb R^n$, we have $\pi^\pm e = e \frac{1}{2}(1 \mp \omega)$. 

\begin{lemma} \label{cliff3}
$$\mathcal Cl(\mathbb R^3) \cong \mathbb H \oplus \mathbb H$$
\end{lemma}

\begin{proof}

Define an algebra isomorphism $\mathbb H \oplus \mathbb H \rightarrow \mathcal Cl(\mathbb R^3)$ by
$$ \begin{array}{l l}
      i \oplus 0 \mapsto \frac{1}{2} (e_1 e_2 - e_3) &
      j \oplus 0 \mapsto \frac{1}{2} (e_2 e_3 - e_1) \\ & \\
      0 \oplus i \mapsto \frac{1}{2} (e_1 e_2 + e_3)  &
      0 \oplus j \mapsto \frac{1}{2} (e_2 e_3 + e_1)  
   \end{array} 
$$
where the $e_i$ are the usual oriented orthonormal basis for $\mathbb R^3 \subset Cl(\mathbb R^3)$.
\end{proof}

Note this map sends $\mathbb H \oplus 0$ to  $Cl^+(\mathbb R^3)$ and $0 \oplus \mathbb H$ to $Cl^-(\mathbb R^3)$.

\begin{lemma} \label{cliffm1}
$$\mathcal C l(K^{n-1}) \cong \mathcal C l_0 (K^n)$$
\end{lemma}

\begin{proof}
We can define an algebra isomorphism $\mathcal C l( K^{n-1}) \rightarrow \mathcal C l_0 ( K^n)$ by
$$ e_i \mapsto e_i e_n$$
where the $e_i$ on the left are the usual orthonormal basis for $ K^{n-1} \subset C l( K^{n-1})$ and the $e_i$ on the right are the usual orthonormal basis for $K^n \subset C l(K^n)$.

\end{proof}

Note the isomorphism from the previous Lemma sends the $\omega$ of $\mathcal C l(\mathbb R^3)$ to the $\omega$ of $\mathcal C l (\mathbb R^4)$ so it preserves the corresponding splittings, i.e. $\mathcal C l^\pm (\mathbb R^3) \mapsto \mathcal C l^\pm_0 (\mathbb R^4)$.

Similarly $C l(\mathbb C^n)$ has a {\it complex volume element} (oriented by the usual orientation of $\mathbb R^n$), $\omega_\mathbb C = i^{\left [ \frac{n(n-1)}{2} \right ]} e_1 \cdots e_n$ where the $e_i$ are an oriented orthonormal basis for $\mathbb R^n = (\mathbb R \oplus 0) \oplus \cdots \oplus (\mathbb R \oplus 0)  \subset \mathbb C^n \subset C l(\mathbb C^n)$. Note $\omega_\mathbb C^2 = 1$ so it induces a splitting
$$ \mathcal C l(\mathbb C^n) = \mathcal C l^+ (\mathbb C^n) \oplus \mathcal C l^-(\mathbb C^n)$$
where $\mathcal C l^\pm (\mathbb C^n) = \pi^\pm_\mathbb C \mathcal C l (\mathbb C^n)$ with $\pi^\pm_\mathbb C = \frac{1}{2}(1 \pm \omega_\mathbb C)$. 

Finally note Clifford multiplication induces a splitting 
$$\mathbb C^4 = (\mathbb C^4)^+ \oplus (\mathbb C^4)^-$$
where $(\mathbb C^4)^\pm = \pi^\pm_\mathbb C \cdot \mathbb C^4$.  

\begin{lemma}\label{cliffiso} $\dim (\mathbb C^4)^\pm = 2$ and Clifford multiplication induces vector space isomorphisms:
$$ \begin{array}{r l}
      \mathbb C^4 \rightarrow  & Hom_\mathbb C ((\mathbb C^4)^+, (\mathbb C^4)^-) \\
      \mathbb C^4 \rightarrow  & Hom_\mathbb C ((\mathbb C^4)^-, (\mathbb C^4)^+).
   \end{array} 
$$
\end{lemma}

\begin{proof}
First define a linear maps $\phi_\pm : \mathbb C^4 \rightarrow Hom_\mathbb C ((\mathbb C^4)^\pm, (\mathbb C^4)^\mp)$ by $e \mapsto (v \mapsto e \cdot v)$ where $e \in \mathbb C^4 \subset \mathcal C l (\mathbb C^4)$ and $v \in (\mathbb C^4)^\pm \subset \mathbb C^4$. 

To see this is well-defined, choose $v \in (\mathbb C^4)^\pm$ and $e \in \mathbb C^4 \subset \mathcal C l (\mathbb C^4)$. Then
$$e \cdot v = e \pi^\pm_\mathbb C \cdot v = \pi^\mp_\mathbb C e \cdot v \in (\mathbb C^4)^\mp.$$

For any $e \in \mathbb C^4 \subset \mathcal C l(\mathbb C^4)$ with $<e,e> \neq 0$ and $v \in (\mathbb C^4)^\pm$, $$\phi_\mp (e) \circ \phi_\pm(e) (v) = -<e,e>v$$ is an automorphism of $(\mathbb C^4)^\pm$ hence $\dim (\mathbb C^4)^+ = \dim (\mathbb C^4)^-$. Then since $\dim (\mathbb C^4)^+ + \dim (\mathbb C^4)^- = 4$, $\dim (\mathbb C^4)^\pm = 2$.

Now we will show $\phi_\pm$ are injective. First suppose $\phi_+(e) = 0$ for some $e \in \mathbb C^4 \subset \mathcal C l(\mathbb C^4)$ then $e \cdot v = 0$ for each $v \in (\mathbb C^4)^+$. It follows $e \pi^+_\mathbb C \cdot v^\prime = 0$ for each $v^\prime \in \mathbb C^4$. Since Clifford multiplication is faithful, we have $e \pi^+_\mathbb C = 0$. Hence $\pi^-_\mathbb C e = 0$. We may write $e = c_1 e_1 + c_2 e_2 + c_3 e_3 + c_4 e_4$ where the $e_i$ are the usual orthonormal basis for $\mathbb C^4$ and the $c_i$ are constants in $\mathbb C$. Then the previous equation may be written
$$\pi^-_\mathbb C (c_1 e_1 + c_2 e_2 + c_3 e_3 + c_4 e_4) = 0.$$
If we distribute, this turns into
$$\frac{1}{2} (c_1 e_1 + c_2 e_2 + c_3 e_3 + c_4 e_4 + c_1 e_2 e_3 e_4 - c_2 e_1 e_3 e_4 + c_3 e_1 e_2 e_4 - c_4 e_1 e_2 e_3 ) = 0.$$
This is a linear combination of distinct basis vectors so each $c_i = 0$ and hence $e = 0$. Proving $\phi_-$ is injective can be done similarly.

Therefore since $\dim Hom_\mathbb C ((\mathbb C^4)^\pm, (\mathbb C^4)^\mp) = 4$, $\phi_\pm$ are isomorphisms. 

\end{proof}

Using this Lemma, we can define maps $C : \mathbb C^4 \otimes (\mathbb C^4)^\pm \rightarrow (\mathbb C^4)^\mp$ which we will also refer to as {\it Clifford multiplication}.

\begin{lemma}\label{endiso} Clifford multiplication induces algebra isomorphisms
$$\mathcal C l^\pm_0(\mathbb C^4) \rightarrow End(\mathbb C^4)^\pm.$$
\end{lemma}

\begin{proof}
Define algebra isomorphisms $\phi_\pm : \mathcal C l^\pm_0(\mathbb C^4) \rightarrow End(\mathbb C^4)^\pm$ by $\varphi \mapsto (v \mapsto \varphi \cdot v)$ where $\varphi \in C l^\pm_0(\mathbb C^4)$ and $v \in (\mathbb C^4)^\pm \subset \mathbb C^4$. Since $\varphi = \pi^\pm_\mathbb C \varphi$ for all $\varphi \in C l^\pm(\mathbb C^4)$, this is well-defined. 

To see $\phi^\pm$ are injective suppose $\varphi \cdot v = 0$ for all $v \in (\mathbb C^4)^\pm$. Then $\varphi \cdot \pi^\pm_\mathbb C \cdot v^\prime = 0$ for each $v^\prime \in \mathbb C^4$. Because $\omega_\mathbb C$ commutes with each element in $\mathcal C l_0(\mathbb C^4)$, we have $\varphi \cdot v^\prime = 0$ and since Clifford multiplication is faithful, $\varphi = 0$. 

Then since $\dim C l^\pm_0(\mathbb C^4) = \dim End(\mathbb C^4)^\pm = 4$, $\phi_\pm$ are vector space isomorphisms. Additionally they are algebra isomorphisms since for each $\varphi \in \mathcal C l^\pm_0(\mathbb C^4)$, 
$$\phi_\pm(\varphi_1\varphi_2)(v) = \varphi_1 \varphi_2 \cdot v = \varphi_1 \cdot (\varphi_2 \cdot v) = \phi_\pm(\varphi_1) \circ \phi_\pm(\varphi_2)(v).$$ 

\end{proof}

\end{section}

\begin{section}{Exterior algebras}

First we define the vector space of {\it alternating tensors} of degree $r$ on $K^n$ to be
$$\begin{array}{r l}
\mathcal \wedge^r(K^n) &= \{ \varphi \in T^r(K^n) \mid \varphi(v_1, \cdots, v_i, \cdots, v_j, \cdots v_r) \\
                       &= -\varphi(v_1, \cdots, v_j, \cdots, v_i, \cdots v_r) \}.
\end{array}$$ 

We have the {\it alternating projection} map $Alt : T^r(K^n) \rightarrow \wedge^r(K^n)$ given by
$$\varphi(v_1, \cdots, v_r) \mapsto \frac{1}{r!} \sum_{\sigma \in S_r} sgn(\sigma) \varphi(v_{\sigma(1)}, \cdots, v_{\sigma(r)} )$$
where $S_r$ denotes the set of permutations of $r$ elements.

\begin{definition}
The {\it exterior algebra} on $K^n$ is defined to be the algebra
$$ \wedge(K^n) = \oplus_{r=0}^{\infty} \wedge^r( K^n ) = K \oplus K^n \oplus \wedge^2 (K^n) \oplus \wedge^3 (K^n) \oplus \cdots$$
whose multiplication $\wedge$ is given by linearly extending the map $\wedge^\ell(K^n) \otimes \wedge^m(K^n) \rightarrow \wedge^{\ell + m}(K^n)$ defined by
$$(\varphi, \phi) \mapsto \frac{(\ell + m)!}{\ell ! m!} Alt(\varphi \otimes \phi).$$
\end{definition}

Note $\wedge(K^n)$ is an associative algebra with multiplicative identity 
$$1 \in K = \wedge^0(K^n) \subset \wedge (K^n).$$

\begin{lemma} \cite{Lee} 
Let $\{e_1, \cdots , e_n \}$ denote a basis for $K^n$. Then $\wedge(K^n)$ is generated as an algebra by the $e_i$ and $1$. It follows $\wedge (K^n)$ is a $2^n$-dimensional vector space with basis
$$\{ 1 \} \cup \{e_{i_1} \wedge e_{i_2} \wedge \cdots \wedge e_{i_r} \mid 1 \leq i_1 < i_2 < \cdots i_r \leq n\}.$$
 \hfill $\square$
\end{lemma}
 
\begin{corollary}\label{extcliffiso}
$\wedge(K^n)$ and $\mathcal C l(K^n)$ are naturally isomorphic as vector spaces. \hfill $\Box$
\end{corollary}

We define the {\it Hodge star operation} on $\wedge (\mathbb R^n)$ (induced by the usual orientation and the usual inner product $<,>$) to be the maps 
$$\ast :\wedge^k(\mathbb R^n) \rightarrow \wedge^{n-k}(\mathbb R^n)$$
given by $\ast (e_{i_1} \wedge \cdots e_{i_k}) = e_{i_{k+1}} \wedge e_{i_n}$ where the $e_i$ are an oriented orthonormal basis and $(i_1, \cdots , i_n)$ is an even permutation of $\{1,2, \cdots, n \}$.

In the $n=4$ case, we have
$$ \begin{array}{l l}
      e_1 \wedge e_2 \mapsto e_3 \wedge e_4 & e_1 \wedge e_3 \mapsto -e_2 \wedge e_4 \\
      e_1 \wedge e_4 \mapsto e_2 \wedge e_3 & e_2 \wedge e_3 \mapsto e_1 \wedge e_4 \\
      e_2 \wedge e_4 \mapsto -e_1 \wedge e_3 & e_3 \wedge e_4 \mapsto e_1 \wedge e_2 \\
   \end{array} 
$$
and $\ast :\wedge^2(\mathbb R^4) \rightarrow \wedge^2(\mathbb R^4)$ induces a splitting
$$ \wedge^2(\mathbb R^4) = \wedge^-(\mathbb R^4) \oplus \wedge^+(\mathbb R^4)$$
where $\wedge^\pm(\mathbb R^4) = \{ \varphi \in \wedge^2(\mathbb R^4) \mid \ast \varphi = \pm \varphi \}$.

\begin{lemma}\label{extcliff}
The natural isomorphism from Corollary~\ref{extcliffiso} induces a vector space isomorphism between the subspaces
$$\mathcal (C l_0(\mathbb R^4) \otimes \mathbb C)^+ \,\,\, \text{and} \,\,\, \pi^+_\mathbb C \mathbb C \oplus (\wedge^+(\mathbb R^4) \otimes \mathbb C).$$
\end{lemma}

\begin{proof}
Let $\{e_1, \cdots, e_4 \}$ denote an oriented orthonormal basis for $\mathbb R^4 = \mathbb R^4 \otimes 1$. Notice $(C l_0(\mathbb R^4) \otimes \mathbb C)^+$ has the basis 
$$ \left \{ \pi^+_\mathbb C, e_1 e_2 + e_3 e_4, e_1 e_3 -e_2 e_4, e_1 e_4 + e_2 e_3 \right \}$$
so we see $(C l_0(\mathbb R^4) \otimes \mathbb C)^+ \mapsto  \pi^+_\mathbb C \mathbb C \oplus \wedge^+(\mathbb R^4)\otimes \mathbb C$.
\end{proof}

\end{section}

\begin{section}{$\bm{ Spin(n)}$ and $\bm{ Spin^c(n)}$}

Now we will define the Lie groups $Spin(n)$ and $Spin^c(n)$. Let $\mathcal C l^\times (\mathbb R^n)$ denote the multiplicative group of units in $\mathcal C l (\mathbb R^n)$. We define $Pin(n)$ to be the subgroup of $\mathcal C l^\times (\mathbb R^n)$ generated by elements $v \in \mathbb R^n \subset C l (\mathbb R^n)$ with $< v, v > = 1$. We define $Spin(n)$ to be the intersection of $Pin(n)$ and $\mathcal C l_0 (\mathbb R^n)$.

\begin{lemma}\label{spin4iso} We have an isomorphism $SU(2) \times SU(2) \rightarrow Spin(4)$ where the splitting on the left corresponds with the splitting $Spin(4) = Spin(4)^+ \times Spin(4)^-$. Here $Spin(4)^\pm$ denotes the splitting of $Spin(4)$ induced by the splitting $C l^\pm (\mathbb R^4)$. 
\end{lemma}

\begin{proof}
Recall we have algebra isomorphisms
$$ \mathbb H \oplus \mathbb H \rightarrow \mathcal C l(\mathbb R^3) \rightarrow \mathcal C l_0(\mathbb R^4)$$
where the first map is the isomorphism from Lemma~\ref{cliff3} and the second map is the isomorphism from Lemma~\ref{cliffm1}.

If we identify $\mathbb H$ with $\mathbb R^4$, the group of unit quaternions is identified with $S^3 \subset \mathbb R^4$ and is isomorphic to $SU(2)$. Then by restricting the composition of the above maps to $S^3 \times S^3 \subset \mathbb H \oplus \mathbb H$, we obtain an isomorphism
$$ SU(2) \times SU(2) \rightarrow Spin(4).$$

This composition preserves the desired splittings as each of the isomorphisms above preserves its own corresponding splittings.  

\end{proof}

We define $Spin^c(n)$ to be the multiplicative group of units $[ Spin(n) \times S^1] \subset \mathcal C l (\mathbb R^n) \otimes \mathbb C$. Observe by using the algebra isomorphism from Lemma~\ref{complexcliff} , we can consider $Spin^c(n)$ to be a multiplicative group of units contained in $\mathcal C l(\mathbb C^4)$.

\begin{lemma}\label{spincspin} $Spin^c(n) \cong  Spin(n) \times S^1 / \{\pm(1,1) \}$.
\end{lemma}

\begin{proof}
First we have the natural surjective group homomorphism 
$$Spin (n) \times S^1 \hookrightarrow [Spin(n) \times S^1].$$
Elements of the kernel of this map are of form $(c1, c^{-1})$ where $c \in S^1 \cap \mathbb R = \{-1, 1 \}$ and $c1 \in Spin(n)$. To see $-1 \in Spin(n) = Pin(n) \cap C l_0(\mathbb R^n)$, first $-1 \in Pin(n)$ since $e_1 e_1 = -1$ and $-1 \in \mathcal C l_0(\mathbb R^n)$ since $$\alpha(-1) = \alpha(e_1 e_1) = \alpha(e_1)\alpha(e_1) = (-e_1)(-e_1) = e_1 e_1 = -1.$$ 
Thus the kernel of our map is $\{\pm(1,1) \}$ so by the first isomorphism theorem, $$Spin^c(n) \cong  Spin(n) \times S^1 / \{\pm(1,1) \}.$$
\end{proof}

From this Lemma and Lemma~\ref{spin4iso}, we obtain an isomorphism $$Spin^c(4) \cong SU(2) \times SU(2) \times S^1 / \{\pm(I,I,1) \}.$$ Note that under this isomorphism, $Spin^c(4)^+ \oplus \pi^-_\mathbb C \subset \mathcal C l(\mathbb C^4)$ is identified with the subgroup $[SU(2) \times I \times S^1] \subset SU(2) \times SU(2) \times S^1 / \{\pm(I,I,1) \}$ and $Spin^c(4)^- \oplus \pi^+_\mathbb C$ is identified with the subgroup $[I \times SU(2) \times S^1]$.

\begin{lemma}\label {spinc3u2} We have a group isomorphism $\mu : Spin^c(3) \rightarrow U(2).$ \hfill $\Box$
\end{lemma}

\begin{lemma}\label {spinc4u2} We have a group isomorphism $$\{(A, B) \in U(2) \times U(2)\ \mid \det(A) = \det(B)\} \rightarrow Spin^c(4) \subset \mathcal C l(\mathbb C^4)$$ where the splitting of $\{(A, B) \in U(2) \times U(2)\ \mid \det(A) = \det(B)\}$ corresponds with the splitting $Spin^c(4) = Spin^c(4)^+ \times Spin^c(4)^-$.
\end{lemma}

\begin{proof}
First there is an isomorphism 
$$\{(A, B) \in U(2) \times U(2)\ \mid \det(A) = \det(B)\} \rightarrow SU(2) \times SU(2) \times S^1 / \{ \pm(I, I, 1)\}$$
defined by
$$(A, B) \mapsto [( A  
\left[ {\begin{array}{ll}
 \lambda^{-1} & 0 \\
 0 & \lambda^{-1}
\end{array} } \right]
, B 
\left[ {\begin{array}{ll}
\lambda^{-1} & 0 \\
 0 & \lambda^{-1}
\end{array} } \right]
, \lambda)]$$ where $\lambda^2 = \det A$. 
Note since 
$$ [(A, B, \lambda )] = [(-A, -B, -\lambda)] \,\,\, \text{ in } \,\,\, SU(2) \times SU(2) \times S^1 / \{ \pm(I, I, 1)\},$$
our map is the same for each of the two choices of $\lambda$ and hence is well-defined.

The rest follows from the comments after Lemma~\ref{spincspin}.

\end{proof}

The {\it adjoint representation} of $Spin(n)$ is the map $Ad : Spin(n) \rightarrow Aut(\mathcal C l(\mathbb R^n))$ defined by $\varphi \mapsto (y \mapsto \varphi y \varphi^{-1})$. Recall 
$$Pin(n) = \{v_1 \cdots v_r \in \mathcal C l(\mathbb R^n) \mid v_i \in \mathbb R^n \text{ with } <v_i, v_i> = 1 \}$$
so we see $Ad_\varphi (v) \in \mathbb R^n$ for each $\varphi \in Spin(n) \subset Pin(n)$ and $v \in \mathbb R^n$. Hence we can restrict the range to obtain a homomorphism $Ad : Spin(n) \rightarrow GL(n)$ . In fact:  

\begin{lemma}\label{spincover} \cite{SpinGeo} $Ad$ induces a group homomorphism, 
$$\xi : Spin(n) \rightarrow SO(n)$$
which is a double covering map. For $n > 2$, this is the universal double cover (up to isomorphism). \hfill $\square$
\end{lemma}

For $Spin^c(n)$, we can define a double-covering map of $SO(n) \times U(1)$ as follows. Let $\xi^c: Spin^c(n) \rightarrow SO(n) \times U(1)$ be the homomorphism $[(\varphi, \lambda)] \mapsto (\xi(\varphi), \lambda^2)$. Also observe the map $\xi : Spin(n) \rightarrow SO(n)$ induces the homomorphism $\xi : Spin^c(n) \rightarrow SO(n)$ given by $[(\varphi, \lambda)] \mapsto \xi(\varphi)$. The kernel of this map is $Z(Spin^c(n)) \cong S^1$.

\begin{lemma}\label{spincs1}
$$Spin^c(n) \cong Spin^c(n) \times_{S^1} S^1 =  Spin^c(n) \times S^1 / \{(\lambda 1, \lambda^{-1}) \mid \lambda \in S^1 \}$$
\end{lemma}

\begin{proof}
Define an isomorphism by $\varphi \mapsto [\varphi, 1]$. To see this is onto observe $[\varphi, \lambda] = [\lambda \varphi, 1]$ for each $\varphi \in Spin^c(n)$ and $\lambda \in S^1$. To see injectivity suppose $[\varphi_1, 1] = [\varphi_2, 1]$ for some $\varphi_i \in Spin^c(n)$. Then $(\varphi_1, 1) = (\lambda \varphi_2, \lambda^{-1})$ for some $\lambda \in S^1$. Hence $\lambda = 1$ and $\varphi_1 = \varphi_2$.
\end{proof}

\end{section}

\begin{section}{$\bm{Spin}$ and $\bm { Spin^c}$-structures}

In this section we will define $Spin$ and $Spin^c$-structures. Given an orientable manifold $X$, recall a choice of orientation and Riemannian metric reduces the structure group of $TX$ to $SO(n) \subset GL(n)$ hence we obtain a frame bundle $P_{SO(n)}$.

\begin{definition}
A $Spin${\it -structure} for an oriented Riemannian $n$-manifold $X$ is a principal $Spin(n)$-bundle $P_{Spin^c(n)} \rightarrow X$ together with a bundle map $P_{Spin(n)} \rightarrow P_{SO(n)}$ that is $\xi \colon Spin(n) \rightarrow SO(n)$ fibrewise (see Lemma \ref{spincover}).
\end{definition}

\begin{definition}
A $Spin^c${\it -structure} for an oriented Riemannian $n$-manifold $X$ is a principal $Spin^c(n)$-bundle $P_{Spin^c(n)} \rightarrow X$ together with a bundle map $P_{Spin^c(n)} \rightarrow P_{SO(n)}$ that is $\xi \colon Spin^c(n) \rightarrow SO(n)$ fibrewise.
\end{definition}

For $n=3$ and $n=4$, the {\it determinant line bundle} of a $Spin^c$-structure $P_{Spin^c(n)} \rightarrow P_{SO(n)}$ is defined to be the complex line bundle $L = P_{Spin^c(n)} \times_{\det} \mathbb C$ where $\det : Spin^c(n) \rightarrow U(1)$ is given by $[(\varphi, \lambda)] \mapsto \lambda^2$.

\begin{lemma} \cite{Morgan} We may equivalently define a $Spin^c${\it -structure} to be a principal $Spin^c(n)$-bundle $P_{Spin^c(n)} \rightarrow X$ together with a bundle map $P_{Spin^c(n)} \rightarrow P_{SO(n)} \oplus P_{U(1)}$ that is $\xi^c \colon Spin^c(n) \rightarrow SO(n) \oplus U(1)$ fibrewise. \hfill $\Box$
\end{lemma}

\begin{lemma} \cite{bald2} For $n = 3$ (or $4$), we may equivalently define a $Spin^c${\it -structure} to be $2$ (or $4$) -dimensional complex vector bundle $W$ endowed with a Hermitian metric and a map $\rho : T^*M \rightarrow End(W)$ satisfying
$$\rho(v)\rho(w) + \rho(w)\rho(v) = -2<v,w>Id_W.$$ \hfill $\Box$   
\end{lemma}

In the four-dimensional case, using the Clifford multiplication map $\mu : \mathcal C l(\mathbb C^4) \rightarrow Mat(\mathbb C, 4)$ from Lemma~\ref{cliffmult}, the {\it complex spinor bundle} associated to $\mu$ is defined to be the complex vector bundle $W = P_{Spin^c(n)} \times_{\mu} \mathbb C^4$.

We may split $W$ as $W = W^+ \oplus W^-$ where 
$$W^{\pm} = P_{Spin^c(4)}\times_{\mu^{\pm}} (\mathbb C^4)^\pm$$
where $\mu^\pm(\bullet) = \mu (\pi^\pm_\mathbb C\bullet)$. 
$W^+$ is called the {\it positive complex spinor bundle} and $W^-$ is called the {\it negative complex spinor bundle}. From Lemma~\ref{spinc4u2}, both $W^\pm$ have structure group $U(2)$.

Similarly in the three-dimensional case, we define the {\it complex spinor bundle} to be 
$W = P_{Spin^c(3)} \times_{\mu} \mathbb C^2.$

Now we will show $H^2(X; \mathbb Z)$ has an action on $Spin^c(X)$ (the set of isomorphism classes of $Spin^c$-structures on $X$). For $E \in H^2(X; \mathbb Z)$, let $P_{U(1)}$ denote the corresponding principal $U(1)$-bundle. We can define a new $Spin^c$-structure $\xi \otimes E$ as follows. Consider
$$ P_{Spin^c(4)} \times_{U(1)} P_{U(1)} = P_{Spin^c(4)} \times P_{U(1)}  / \sim $$
where $(\varphi, y) \sim (\varphi \cdot \lambda, y \cdot \lambda^{-1})$ for each $\lambda \in U(1)$. On the left, $U(1)$ is identified with $Z(Spin^c(n))$ in our usual way. From Lemma~\ref{spincs1}, this is a principal $Spin^c(n)$ bundle. We can define our bundle map $P_{Spin^c(4)} \times_{U(1)} P_{U(1)} \rightarrow P_{SO(n)}$ by $[\varphi, y] \mapsto \xi(\varphi)$.

\begin{lemma} \cite{4ManifoldsKirby} The above action is free and transitive. \hfill $\Box$
\end{lemma}
 
Observe the induced map $\det : Spin^c(n) \times_{S^1} S^1 \rightarrow S^1$ is given by $[\varphi \otimes z, \lambda] \mapsto z^2 \lambda^2$. We can write this as $\det = \det_1 \det_2 \det_2$ where $\det_i : Spin^c(n) \times_{S^1} S^1 \rightarrow S^1$ are given by $det_1([\varphi \otimes z, \lambda]) = z^2$ and $det_2([\varphi \otimes z, \lambda]) = \lambda$. Hence
$$
\begin{array}{r l}
(P_{Spin^c(n)} \times_{U(1)} P_{U(1)}) \times_{\det} \mathbb C =&  (P_{Spin^c(n)} \times_{U(1)} P_{U(1)}) \times_{\det} \mathbb C \otimes \mathbb C \otimes \mathbb C  \\
 = & ((P_{Spin^c(n)} \times_{U(1)} P_{U(1)}) \times_{\det_1} \mathbb C) \otimes \\
 & ((P_{Spin^c(n)} \times_{U(1)} P_{U(1)}) \times_{\det_2} \mathbb C) \otimes \\
 & ((P_{Spin^c(n)} \times_{U(1)} P_{U(1)}) \times_{\det_2} \mathbb C) \\
 =& L \otimes E \otimes E
\end{array}
$$
So we see our action has the following effect on determinant line bundles: $L \mapsto L + 2E$.

In the $n=4$ case, observe the induced map 
$$\mu : Spin^c(4) \times_{S^1} S^1 \rightarrow Mat(\mathbb C, 4)$$
is given by $[\varphi \otimes z, \lambda] \mapsto \lambda \mu(\varphi \otimes z)$ which we will write as $\mu = \mu_1 \mu_2$. So
$$
\begin{array}{r l}
(P_{Spin^c(4)} \times_{U(1)} P_{U(1)}) \times_{\mu} \mathbb C^4 =&  (P_{Spin^c(4)} \times_{U(1)} P_{U(1)}) \times_{\mu} \mathbb C^4 \otimes \mathbb C  \\
 = & ((P_{Spin^c(4)} \times_{U(1)} P_{U(1)}) \times_{\mu_2} \mathbb C^4) \otimes \\
 & ((P_{Spin^c(4)} \times_{U(1)} P_{U(1)}) \times_{\mu_1} \mathbb C) \\
 =& W \otimes E 
\end{array}
$$
Thus we see our action has the following effect on complex spinor bundles: $W \mapsto W \otimes E$.

The $n=3$ case is similar and we also have $W \mapsto W \otimes E$.

\end{section}

\begin{section}{Clifford bundles}

\begin{definition}
Given a oriented Riemannian $n$-manifold $X$ with frame bundle $P_{SO(n)}$, we define the {\it Clifford bundle} of $X$ as $\mathcal C l (X) = P_{SO(n)} \times_{SO(n)} \mathcal C l (\mathbb R^n)$. We also have the complexified Clifford bundle $C l (X) \otimes \mathbb C = P_{SO(n)} \times_{SO(n)} (\mathcal C l (\mathbb R^n) \otimes \mathbb C)$.
\end{definition}

Let $X$ be a oriented Riemannian $n$-manifold with frame bundle $P_{SO(n)}$ and $Spin^c$-structure $\xi : P_{Spin^c(n)} \rightarrow P_{SO(n)}$. 

\begin{lemma}
The map $\xi : P_{Spin^c(n)} \rightarrow P_{SO(n)}$ induces a bundle isomorphism:
$$P_{Spin^c(n)} \times_{Ad} \mathcal C l (\mathbb R^n) \otimes \mathbb C \rightarrow  \mathcal C l(X) \otimes \mathbb C$$
where $Ad : Spin^c(n) \rightarrow Aut(\mathcal C l (\mathbb R^n) \otimes \mathbb C)$ is given by 
$$\varphi \otimes \lambda \mapsto (y \otimes v \mapsto \varphi y \varphi^{-1} \otimes \lambda v \lambda^{-1} = \varphi y \varphi^{-1} \otimes v).$$
\end{lemma}

\begin{proof}
Define a map
$$P_{Spin^c(n)} \times \mathcal C l (\mathbb R^n) \otimes \mathbb C \rightarrow  P_{SO(n)} \times \mathcal C l(\mathbb R^n) \otimes \mathbb C$$
by $(y, v) \mapsto (\xi(y), v)$. For $\varphi \otimes \lambda \in Spin^c(n)$ and $(y,v\otimes z) \in P_{Spin^c(n)} \times \mathcal C l (\mathbb R^n) \otimes \mathbb C$, we have
$$(y \cdot \varphi^{-1} \otimes \lambda^{-1}, \varphi v \varphi^{-1} \otimes z) \mapsto (\xi(y) \cdot \xi(\varphi)^{-1}, \xi(\varphi)v \otimes z)$$
so our map induces a bundle map
$$\xi^\prime : P_{Spin^c(n)} \times_{Ad} \mathcal C l (\mathbb R^4) \otimes \mathbb C \rightarrow  \mathcal C l(X) \otimes \mathbb C.$$

Surjectivity follows from the fact that $\xi$ is onto. To see $\xi^\prime$ is injective suppose 
$$\xi^\prime([y_1, v_1\otimes z_1 ]) = \xi ^\prime([y_2, v_2\otimes z_2 ]).$$

Then $[(\xi(y_1), v_1\otimes z_1)] = [(\xi(y_2), v_2 \otimes z_2 )]$ and hence  
$$(\xi(y_1) \cdot \xi(\varphi)^{-1}, \xi(\varphi)v_1 \otimes z_1) = (\xi(y_2), v_2 \otimes z_2)$$
for some $\varphi \otimes \lambda \in Spin^c(n)$. Since $Spin^c(n)$ acts transitively on the fibres of $P_{Spin^c(n)}$, we have $y_1 \cdot (\varphi^\prime \otimes \lambda^\prime)^{-1} = y_2$ for some $\varphi^\prime \otimes \lambda^\prime \in Spin^c(4)$. Observe $$\xi(y_2) = \xi(y_1 \cdot (\varphi^\prime \otimes \lambda^\prime)^{-1}) = \xi(y_1) \cdot \xi(\varphi^\prime)^{-1}$$ so since $SO(n)$ acts freely on the fibres of $P_{SO(n)}$, it follows $\xi(\varphi^\prime) = \xi (\varphi)$ and hence
$$(y_1 \cdot (\varphi^\prime \otimes \lambda^\prime)^{-1}, \xi(\varphi^\prime) v_1 \otimes z_1) = (y_2, v_2 \otimes z_2)$$ 
and therefore $\xi^\prime$ is a bundle isomorphism.

\end{proof}

Now additionally suppose $X$ is $4$-dimensional with complex spinor bundle $W = W^+ \oplus W^-$. We'll show $C l (X) \otimes \mathbb C$ has an action called Clifford multiplication on $W$. Define a map 
$$C : P_{Spin^c(4)} \times (\mathcal C l (\mathbb C^4) \otimes \mathbb C^4) \rightarrow P_{Spin^c(4)} \times \mathbb C^4$$
by $(q, \varphi \otimes v) \mapsto (q, \varphi \cdot v)$ where $\cdot$ denotes Clifford multiplication. For $g \in Spin^c(4)$, we have 
$$C(q g^{-1}, g \varphi g^{-1} \otimes g \cdot v) = (q g^{-1}, g \varphi g^{-1} \cdot g \cdot v) = (q g^{-1}, g \cdot (\varphi \cdot v))$$
so this induces a bundle map
$$C: (\mathcal C l(X) \otimes \mathbb C) \otimes W \rightarrow W$$ 
which we will refer to as the {\it Clifford multiplication} map yet again.

Finally from Lemma~\ref{cliffembed}, $\mathcal C l (X)$ contains the subbundle 
$$P_{SO(4)} \times_{SO(4)} \mathbb R^4 \subset P_{SO(4)} \times_{SO(4)} \mathbb \mathcal C l(\mathbb R^4) = C l (X)$$ which is canonically isomorphic to $TX$. It follows $C l (X) \otimes \mathbb C$ contains a subbundle canonically isomorphic to $TX \otimes \mathbb C$. Thus using the canonical identification of tangent and cotangent bundles, we can define a map
$$C: (T^*X \otimes \mathbb C) \otimes W \rightarrow W.$$
As a result of Lemma~\ref{cliffiso} , we have the restrictions
$$C: (T^*X \otimes \mathbb C) \otimes W^\pm \rightarrow W^\mp.$$

\end{section}

 \chapter{The Seiberg-Witten Invariants of 4-Manifolds}

In this chapter, we will define the Seiberg-Witten invariants of an oriented Riemannian $4$-manifold. Most of these results can be found in \cite{Morgan} and \cite{4ManifoldsKirby}.

\begin{section}{The Dirac operator}

There are two maps necessary for the Seiberg-Witten equations, the Dirac operator and the curvature map. Here we define the Dirac operator.

Let $X$ be a oriented Riemannian 4-manifold with metric $g$. Consider a $Spin^c(4)$ structure on $X$ with determinant line bundle $L$ and complex spinor bundles $W^{\pm}$.

Recall the Levi-Civita connexion on $X$ is an $SO(4)$-connexion on the vector bundle $TX \rightarrow X$ induced by $g$, so we may recover a principal connexion on the frame bundle of $X$, $P_{SO(4)}$.
Let $A \in \mathcal{A}_L$ where $\mathcal{A}_L$ is the space of $U(1)$-connexions on $L$. From this, we may recover a principal connexion on the frame bundle of $L$, $P_{U(1)}$. 
Together, these two connexions determine a principal connexion on $P_{SO(4)} \oplus P_{U(1)}$. Note $Lie(Spin^c(4)) \cong Lie(SO(4)) \oplus Lie(U(1))$ so we may pull back this connexion to a principal connexion on $P_{Spin^c(4)}$ via the given bundle map $P_{Spin^c(4)} \rightarrow P_{SO(4)} \oplus P_{U(1)}$.

$$
\begin{tikzpicture}[description/.style={fill=white,inner sep=2pt}]
\matrix (m) [matrix of math nodes, row sep=3em,
column sep=2.5em, text height=1.5ex, text depth=0.25ex]
{       & P_{Spin^c(4)}              & W = W^+ \oplus W^-   \\
    TX    & P_{SO(4)} \oplus P_{U(1)}  & P_{U(1)} & L  \\};
\path[->,font=\scriptsize]
(m-1-2) edge node[auto]{} (m-2-2);
\path[->,font=\scriptsize]
(m-2-2) edge node[auto]{$proj_2$} (m-2-3);
\end{tikzpicture}
$$
We point out that $P_{Spin^c(4)}$ is the frame bundle of $W$ and $P_{SO(4)}$ is the frame bundle of $TX$. We write the induced connexion on $W^+$ as $\nabla_A \colon \Gamma (X ; W^+ ) \rightarrow \Gamma (X ; W^+ \otimes T^* X )$. 

The {\it Dirac operator} $\slashed{\partial}_A \colon \Gamma (X ; W^+) \rightarrow \Gamma (X ; W^-)$ induced by $A \in \mathcal{A}_L$ is defined to be the composition of the two maps, $\slashed{\partial}_A = C \circ \nabla_A$ where $C \colon \Gamma (X, W^+ \otimes T^* X) \rightarrow \Gamma (X, W^-)$ is the Clifford multiplication map.

\end{section}

\begin{section}{The curvature map}

Now we define the curvature map. Let $X$ be an oriented Riemannian $4$-manifold together with a $Spin^c(4)$-structure $\xi$ having positive complex spinor bundle $W^+$. The following is a standard result that can be found in \cite{found}.

\begin{proposition} If $E \rightarrow X$ is an $n$-dimensional complex vector bundle then there is a natural one-to-one correspondence between Hermitian metrics on $E$ and reductions of the structure group of $E$ to $U(n) \subset GL(\mathbb C, n)$. \hfill $\Box$
\end{proposition}

So since we already have a reduction of the structure group of $W^+$ to $U(2)$ via the associated bundle construction, we also have a Hermitian metric $h : W^+ \times W^+ \rightarrow \mathbb C$. Note $h$ induces an anti-complex bundle isomorphism $W^+ \rightarrow (W^+)^*$.

Consider the map,
$$ \begin{array}{r l}
 q : W^+ & \rightarrow End(W^+) \\
                                        
                   \varphi & \mapsto \varphi \otimes h(\varphi, \bullet) - \frac{1}{2} (\text{Tr}(\varphi \otimes h(\varphi, \bullet))Id.
   \end{array}
$$

From Lemma \ref{endiso}, Clifford multiplication induces a vector bundle isomorphism
$$End(W^+) \rightarrow P_{SO(4)} \times_{SO(4)} \mathcal C l^+_0(\mathbb C^4)$$
and from Lemma \ref{extcliff}, we have an isomorphism
$$P_{SO(4)} \times_{SO(4)} \mathcal C l^+_0(\mathbb C^4) \rightarrow P_{SO(4)} \times_{SO(4)} \pi_\mathbb C^+ \mathbb C \oplus (\wedge^+(\mathbb R^4) \otimes \mathbb C).$$

Observe when $C l^+_0(\mathbb C^4)$ acts on $(\mathbb C^4)^+$ by Clifford multiplication, $\pi_\mathbb C^+$ acts as the identity. Thus since all the endomorphisms in the image of $q$ are traceless, by composing these bundle maps we obtain the map
$$ q : W^+ \rightarrow P_{SO(4)} \times_{SO(4)} \wedge^+(\mathbb R^4) \otimes \mathbb C = \wedge^+(TX) \otimes \mathbb C$$

Finally using the canonical identification of tangent and cotangent bundles, we can define the map
$$ q : W^+ \rightarrow \wedge^+(X) \otimes \mathbb C$$

\begin{lemma}
\cite{Morgan} For each $\sigma \in \Gamma(W^+)$, $q(\sigma) \in i\Omega^+(X)$ is a purely imaginary self-dual $2$-form. \hfill $\Box$
\end{lemma}

\end{section}

\begin{section}{The Seiberg-Witten equations and moduli space}

Let $X$ be an oriented Riemannian 4-manifold. Fix a $Spin^c(4)$ structure \\ $\xi : P_{Spin^c(4)} \rightarrow P_{SO(4)}$, determinant line bundle $L$, and complex spinor bundles $W^\pm$.

Consider the space $\mathcal C_\xi := \Gamma (W^+) \times \mathcal A_L$. The {\it Seiberg-Witten equations} are given by
$$ \slashed{\partial}_A \psi = 0 \text{ and } F_A^+ = q(\psi)$$
where $(\psi, A)  \in \Gamma (W^+) \times \mathcal A_L$. Recall the curvature of $A$, $F_A$, can be considered as a $Lie(U(1))$-valued two-form. Since $Lie(U(1)) \cong i \mathbb{R}$, $F_A \in i\Omega^2(X)$. Note $F_A = i\alpha$ for some $2$-form $\alpha \in \Omega^2(X)$ and $\alpha$ splits into self-dual and anti-self-dual parts $\alpha = \alpha^+ + \alpha^-$. Above $F_A^+ = i\alpha^+ \in i\Omega^+(X)$.  

We define the {\it gauge group} $\mathcal G$ of $\xi$ to be the group of automorphisms of $P_{Spin^c(4)}$ that are $id_{Spin^c(4)}$ fibrewise and cover the identity of $P_{SO(4)}$.

\begin{lemma}
We have the following group isomorphism,
$$\mathcal G \cong \text{Map}(X, S^1)$$
where $\text{Map}(X, S^1)$ is endowed with the group structure given by pointwise product on the target.
\end{lemma}

\begin{proof}

Let $g \in \mathcal G$, then we have the following commutative diagram.
$$
\begin{tikzpicture}[description/.style={fill=white,inner sep=2pt}]
\matrix (m) [matrix of math nodes, row sep=3em,
column sep=2.5em, text height=1.5ex, text depth=0.25ex]
{ P_{Spin^c(4)}              & P_{Spin^c(4)}   \\
  P_{SO(4)}   & P_{SO(4)}   \\};
\path[->,font=\scriptsize]
(m-1-1) edge node[auto]{$g$} (m-1-2);
\path[->,font=\scriptsize]
(m-1-1) edge node[auto]{$\xi$} (m-2-1);
\path[->,font=\scriptsize]
(m-1-2) edge node[auto]{$\xi$} (m-2-2);
\path[->,font=\scriptsize]
(m-2-1) edge node[auto]{$id$} (m-2-2);
\end{tikzpicture}
$$
Choose $y \in P_{Spin^c(4)}$. Since $Spin^c(4)$ acts transitively on the fibres of $P_{Spin^c(4)}$, we have $g(y) = y \cdot \varphi$ for some $\varphi \in Spin^c(4)$. Then 
$$\xi(y) = \xi (y \cdot \varphi) = \xi (y) \cdot \xi(\varphi).$$
Since $SO(4)$ acts freely on the fibres of $P_{SO(4)}$, we have $\xi(\varphi) = I$.
The kernel of this map is the center of $Spin^c(4)$, $1 \otimes S^1 \subset \mathcal C l(\mathbb R^4) \otimes \mathbb C$ which is isomorphic to $S^1$ via the map $\varphi \otimes \lambda \mapsto \lambda \varphi$ (since the center of $Spin(4)$ is $\pm 1$, this makes sense).

Now for each $y^\prime$ in the same fibre as $y$, we have $y^\prime = y \cdot \varphi^\prime$ for some $\varphi^\prime \in Spin^c(4)$. But then 
$$g(y^\prime) = g(y \cdot \varphi^\prime) = g(y) \cdot \varphi^\prime = y \cdot \varphi \varphi^\prime = y \cdot \varphi^\prime \varphi = y^\prime \cdot \varphi$$
so we see $g$ is given on this fibre by right multiplication by $\varphi$.

Therefore each automorphism in $\mathcal G$ is given by a (smoothly varying) choice of an element of $S^1 =Z(Spin^c(4))$ for each fibre and on the other hand a map $X \rightarrow S^1$ will determine an automorphism of $P_{Spin^c(4)}$ with the desired properties.

To see this induces a group isomorphism, consider $g,h \in \mathcal G$. Recall for $y \in P_{Spin^c(4)}$, we have $g(y) = y \cdot \varphi_1$ and $h(y^\prime) = y \cdot \varphi_2$ for some $\varphi_i \in Z(Spin^c(4)) = S^1$. Then 
$$g \circ h(y) = g(y \cdot \varphi_2) = y \cdot \varphi_2 \varphi_1$$
since earlier we showed $g(y^\prime) = y^\prime \cdot \varphi_1$ for each $y^\prime$ sharing a fibre with $y$. This is precisely the pointwise product on $\text{Map}(X, S^1)$ so we see the two groups are isomorphic.
\end{proof}

Next we will show the gauge group $\mathcal G$ has an action on $\mathcal{A}_L \times \Gamma (W^+)$. For the remainder of this section, we will consider $\mathcal G$ as the group $Map(X, S^1)$ where $S^1$ is identified with the center of $Spin^c(4)$ via the isomorphism $\varphi \otimes \lambda \mapsto \lambda\varphi$.

Let $g \in \mathcal G$ then for each $y \in P_{Spin^c(4)}$ (where $y$ is in the fibre over some $x \in X$), $v \in (\mathbb C^4)^+$, and $z \in \mathbb C$, we have  
$$ \begin{array}{l l}
  (y \cdot g(x) \varphi^{-1}, \mu^+(\varphi) v) &= (y \cdot \varphi^{-1} g(x) , \mu^+(\varphi) v) \\ 
  (y \cdot g(x) \varphi^{-1}, \det (\varphi) z) &= (y \cdot \varphi^{-1} g(x) , \det (\varphi) z) \\ 
 
   \end{array}
$$
for each $\varphi \in Spin^c(4)$ so $\mathcal G$ induces actions on $W^+$ and $L$. Note
$$ \begin{array}{r l}
  [ (y \cdot g(x) ,  v) ] &= [ (y  , \mu^+ (g(x)) v) ] \\ 
 \,  [ (y \cdot g(x) ,  z) ] &=  [ (y  , \det (g(x)) z) ] 
   \end{array}
$$
and each $\varphi \otimes \lambda \in S^1 \subset Spin^c(4)$ acts on $(\mathbb C^4)^+$ via $\mu^+$ by scalar multiplication by $\lambda \varphi \in S^1 \subset \mathbb C$ and on $\mathbb C$ via $\det$ by scalar multiplication by $\lambda^2$. Thus $\mathcal G$ acts on $\Gamma (W^+)$ by $\sigma(x) \mapsto g(x) \sigma(x)$ and on $\Gamma (L)$ by $\sigma(x) \mapsto g^{2}(x) \sigma(x)$ for each $x \in X$ through fibrewise scalar multiplication.

\begin{remark} For purely aesthetic reasons, we will have $\mathcal G$ act on $P_{Spin^c(4)}$ by {\bf anti-multiplication}, i.e. 
$g(x) \mapsto  (\varphi \mapsto \varphi \cdot g^{-1}(x))$
for each $g \in \mathcal G$, $x \in X$, and $\varphi \in P_{Spin^c(4)}$ living in the fibre over $x$. 
So to summarize, we have the induced fibrewise multiplication actions
$$ \begin{array}{r l l}
  g \mapsto & (\sigma \mapsto g^{-1} \sigma) & \text{for $\sigma \in \Gamma(W^+)$}  \\ 
  g \mapsto & (\sigma \mapsto g^{-2} \sigma) & \text{for $\sigma \in \Gamma(L)$}\\  
   \end{array}
$$
\end{remark}

The last action induces a pullback action on $\mathcal{A}_L$ where a $U(1)$-connexion $$A \colon \Gamma(TX) \times \Gamma (L) \rightarrow \Gamma(L)$$
on $L$ is acted upon by $\mathcal G$ via
$ A(V, \sigma) \mapsto g^{-2} A(V, g^2  \sigma).$
Using the Leibniz rule, we may rewrite $g^{-2}  A(V, g^2 \sigma)$ as 
$$g^{-2}  (g^2  A(V, \sigma) + d(g^2) \sigma) = A(V, \sigma) + 2g^{-1} dg  \sigma$$
so we see the above action is given by $A \mapsto A + 2g^{-1}dg.$

\begin{lemma} \label{gequi}
For $g \in \mathcal G = Map(X, S^1)$ and $(A, \psi) \in \mathcal A_L \times \Gamma (W^+)$, we have
$$   \slashed{\partial}_{(g^{-2})^* A} (g^{-1} \psi)  = g^{-1} \slashed{\partial}_{A} (\psi) \text{ , } 
  F_{(g^{-2})^* A}  = F_A \text {, and }
  q(g^{-1} \psi)  = q(\psi).
$$

\end{lemma}

\begin{proof}
First to see $q(g^{-1}  \psi)  = q(\psi)$, recall the map
$$ \begin{array}{r l}
 q : W^+ & \rightarrow End(W^+) \\
                                       
                   \varphi & \mapsto \varphi \otimes h(\varphi, \bullet) - \frac{1}{2} (\text{Tr}(\varphi \otimes h(\varphi, \bullet))Id.
   \end{array}
$$
which is the first map in the series of compositions in section 2 that defines the curvature map. Restricting ourselves to a fibre, observe for $\lambda \in S^1$ and $\varphi \in W^+$,
$$ \begin{array}{r l}
q(\lambda \varphi) &=  \lambda \varphi \otimes h(\lambda \varphi, \bullet) -\frac{1}{2} (\text{Tr}(\lambda \varphi \otimes h(\lambda \varphi, \bullet))Id \\
 &= \lambda \overline{\lambda} \varphi \otimes h(\varphi, \bullet) -\frac{1}{2} (\text{Tr}(\lambda \overline{\lambda} \varphi \otimes h(\varphi, \bullet))Id \\
 &= \varphi \otimes h(\varphi, \bullet) -\frac{1}{2} (\text{Tr}(\varphi \otimes h(\varphi, \bullet))Id \\
 &=q(\varphi)
   \end{array}
$$ 
so we see $q(g^{-1} \psi)  = q(\psi)$.

Now we'll show $\slashed{\partial}_{(g^{-2})^*A} (g^{-1} \cdot \psi)  = g^{-1} \cdot \slashed{\partial}_{A} (\psi)$. The map $A \mapsto \nabla_A$ is equivariant with respect to our $\mathcal G$ actions so
$$\nabla_{(g^{-2})^* A}(\bullet) = (g^{-1})^* \nabla_A(\bullet) = g^{-1} \cdot \nabla_A(g \cdot \bullet).$$
Thus for each $\psi \in \Gamma(W^+)$, we have
$\nabla_{(g^{-2})^* A}(g^{-1} \cdot \psi) = g^{-1} \cdot \nabla_A(\psi)$
and since Clifford multiplication commutes with the action of $\mathcal G$,
$$\slashed{\partial}_{(g^{-2})^*A} (g^{-1} \cdot \psi) = C(g^{-1} \cdot \nabla_A(\psi)) = g^{-1} \cdot C \circ \nabla_A(\psi) = g^{-1} \cdot \slashed{\partial}_{A} (\psi).$$   

Finally observe $\mathcal G$ has an action on the space of $Lie(U(1))$-valued 2-forms on $X$ defined fibrewise by $g(x) \mapsto (a \otimes \omega \mapsto ad_{g^{-2} (x)}(a) \otimes \omega)$. Note since $U(1)$ is abelian, this action is trivial. The map $A \mapsto F_A$ is equivariant with respect to this action and our action on $\mathcal A_L$ so $F_{(g^{-2})^* A}  = F_A$.

\end{proof}

It follows from Lemma \ref{gequi} that the solution set of the Seiberg-Witten equations descends to a subspace of $\mathcal B_\xi := \mathcal{A}_L \times \Gamma (W^+)  / \mathcal G$. For reasons that will become apparent later, we also define $\mathcal B^*_\xi \subset \mathcal B_\xi$ to be the subspace $\{(A, \psi) \mid \psi  \not \equiv 0 \} / \mathcal G$.

\begin{remark}
For analytic reasons, we will now consider certain Sobolev Completions of the spaces $\mathcal{A}_L \times \Gamma (X, W^+)$ and $\mathcal G$. One can show the quotient of appropriately chosen completions of $\mathcal{A}_L \times \Gamma (X, W^+)$ and $\mathcal G$ is a Banach manifold, thus the usual theorem of analysis (i.e.\ the Implicit Function Theorem) generalize to this infinite dimensional setting. Henceforth, we will consider $\mathcal B_\xi$ and $\mathcal B^*_\xi$ to be quotients of the appropriately chosen completions. See \cite{4ManifoldsKirby}.
\end{remark}

We will denote the solution set of the equations as $\mathcal M_\xi \subset \mathcal B_\xi$ and refer to it as the {\it moduli space} of the equations. 
$\mathcal M_\xi$ may not be a smooth manifold but if we require $b^+_2 (X) > 0$ and choose a generic perturbation $(g, \delta )$ where $g$ is a Riemannian metric on $X$ and $\delta \in \Omega^+_g(X)$ then the solution set $\mathcal M^\delta _\xi (g)$ of the {\it perturbed Seiberg-Witten equations}
$$ \slashed{\partial}_A \psi = 0 \text{ and } F_A^+ + i\delta = q (\psi)$$
gives rise to a smooth manifold.

Observe the set of perturbations is a subset of $Met(X) \times \Omega^2(X)$ which in turn is a subset of $C^\infty(TX \oplus TX, \mathbb R) \times C^\infty(TX \otimes TX, \mathbb R)$ so if we endow the later with the compact-open topology on each factor, we obtain a topology for the set of perturbations. By {\it generic}, we mean a perturbation chosen outside of a particular meagre subset. For details, we refer the reader to \cite{Morgan}.

Note the perturbed equations' solution set descends to a subspace of $\mathcal B_\xi(g)$ as a result of Lemma \ref{gequi}.
\end{section}

\begin{section}{The Seiberg-Witten invariants}

\begin{theorem}
(Seiberg-Witten) \cite{4ManifoldsKirby} Let $X$ be an oriented, closed 4-manifold with $b^+_2 (X) > 0$. Fix a $Spin^c(4)$-structure $\xi$ with determinant line bundle $L$. For a generic metric $g$ and perturbation $\delta \in \Omega^+_g (X)$, the moduli space $\mathcal M^\delta_\xi (g)$ is a smooth, closed submanifold of $\mathcal B^*_\xi$ of dimension
$$ d = \frac{1}{4} (c_1(L)^2 - (3 \sigma (X) + 2\chi (X))).$$ 
Furthermore, a homology orientation (an orientation of the vector space $H^0 (X; \mathbb R) \oplus H^1 (X; \mathbb R) \oplus H^+ (X; \mathbb R)$ ) determines an orientation of $\mathcal M^\delta_\xi (g)$. If $b^+_2 (X) > 1$, the homology class $[\mathcal M^\delta_\xi (g)]$ is independent of the choice of $g$ and $\delta$. \hfill $\Box$
\end{theorem}

For each $x_0 \in X$, we define the {\it based gauge space} $\mathcal G^0_{x_0} \subset \mathcal G$ to be the subgroup $\{ g \in \mathcal G \mid g(x_0) = 1 \} \subset \mathcal G$. 

\begin{lemma} \cite{Morgan} The quotient map, $$\mathcal C^*_\xi / \mathcal G^0_{x_0} \rightarrow C^*_\xi / \mathcal G = B^*_\xi$$ makes $\mathcal C^*_\xi / \mathcal G^0_{x_0}$ into a principal $U(1)$-bundle over $B^*_\xi$ whose isomorphism class is independent of the choice of $x_0$. \hfill $\Box$ 
\end{lemma}

We let $\mu \in H^2 (\mathcal B^*_\xi; \mathbb Z)$ denote the first Chern class of this bundle isomorphism class.

\begin{subsection}{$\bm{b^+_2 > 1}$ manifolds} 

\begin{definition} Let $X$ be an oriented, closed, and smooth 4-manifold with $b^+_2 (X) > 1$ together with a homology orientation $\alpha_X$ and a generic perturbation $(g, \delta)$. The {\it Seiberg-Witten invariant} of $X$ is the map $$SW_{X, \alpha_X} \colon Spin^c(X) \rightarrow \mathbb Z$$ defined as follows. 
For each $\xi \in Spin^c(X)$:
\begin{itemize}
\item If dim $M^\delta_\xi (g)$ is even then $SW_X (\xi) = <\mu^m, [\mathcal M^\delta_\xi (g)]>$ where dim $M^\delta_\xi (g) = 2m$. 
\item If dim $M^\delta_\xi (g)$ is odd, we define $SW_X (\xi) = 0$.
\end{itemize}

\end{definition}

\begin{theorem} (Seiberg-Witten) \cite{4ManifoldsKirby} The Seiberg-Witten function $SW_{X,\alpha_X} \colon Spin^c(X) \rightarrow \mathbb Z$ doesn't depend on the choice of generic metric $g$ or generic perturbation $\delta$. 

For an orientation-preserving diffeomorphism $f \colon X \rightarrow X^\prime$, we have 
$$SW_{X^ \prime,\alpha_{X^\prime}} (\xi) = \pm SW_{X,\alpha_X} (f^*\xi)$$
 where $f^*\xi$ is the induced $Spin^c(4)$-structure on $X$ and $\alpha_X$ is the induced homology orientation on $X$.\hfill $\Box$
\end{theorem}

\end{subsection}

\begin{subsection}{$\bm{b_2^+ = 1}$ manifolds}

We can also define Seiberg-Witten invariants for $b_2^+ = 1$ manifolds. This will be the case we are primarily interested in. Let $X$ be a closed oriented $4$-manifold with $b_2^+(X) = 1$.

\begin{lemma}
\cite{bplus} For each Riemannian metric $g$, there exists a $g$-self-dual harmonic two-form $\omega_g \in \Omega_g^+(X)$ with $[\omega_g]^2 = 1$. \hfill $\Box$
\end{lemma} 

\begin{lemma} \cite{4ManifoldsKirby} $\text{dim } H^+(X; \mathbb R) = b_2^+(X) = 1$. \hfill $\Box$
\end{lemma}

Here $H^+(X; \mathbb R)$ denotes the first factor of the induced splitting of $H^2(X; \mathbb R)$ by the Hodge splitting of $\Omega^2(X)$. As a result of this last lemma and the bijective correspondence between harmonic forms and cohomology elements, our form $\omega_g$ is determined by $g$ up to sign. If we choose a forward cone,
i.e. one of the two connected components of $\{ \Phi \in H^2(X; \mathbb R) \mid [\Phi]^2 > 0\}$, we can fix $\omega_g$ by taking the form whose cohomology class lies in the forward cone. Note a homology orientation for $X$ also determines a forward cone of $H^2(X; \mathbb R)$ hence a choice of $\omega_g$.

Now let $\xi$ be a $Spin^c(4)$-structure for $X$ with determinant line bundle $L$ and complex spinor bundles $W^\pm$.

\begin{definition}
A solution $(A, \psi ) \in \mathcal{A}_L \times \Gamma (W^+)$ to the (perturbed) Seiberg-Witten equations is called {\it reducible} if $\psi \equiv 0$.
\end{definition}

\begin{proposition} 
\cite{4ManifoldsKirby} If the (perturbed) Seiberg-Witten equations admit reducible solutions, $\mathcal G$ won't act freely the solution set of the equations and the moduli space will fail to be a smooth manifold. \hfill $\square$
\end{proposition}

\begin{lemma}
For a given perturbation $(g, \delta )$ where $g$ is a Riemannian metric on $X$ and $\delta \in \Omega^+_g(X)$, there exists a reducible solution of the perturbed Seiberg-Witten equations if and only if 
$$(c_1(L) - \frac{1}{2 \pi} \delta) \cdot \omega_g = 0.$$ \hfill $\Box$
\end{lemma}

\begin{proof}

Suppose the equations admit a reducible solution $(A, \psi)$. This is true if and only if $F_A^+ = \delta$. For each $A \in \mathcal A_L$, $F_A$ represents $\frac{2 \pi}{i} c_1(L) \in i H^2(X; \mathbb R)$ so the equations admit reducible solutions only when $[\delta] = \frac{2 \pi}{i} c_1(L)^+ \in H^+(X; \mathbb R)$. On the other hand, if $[\delta] = \frac{2 \pi}{i} c_1(L)^+$ then there is some $A$ with $F_A^+ = \delta$ so we see the equations admit reducible solutions if and only if $[\delta] = \frac{2 \pi}{i} c_1(L)^+$.

$$
\begin{array}{r l l}
[\delta] = & \frac{2 \pi}{i} c_1(L)^+ & \iff \\
 c_1(L)^+ = & \frac{i}{2 \pi} [\delta] & \iff \\
(c_1(L) - \frac{i}{2 \pi} \delta)^+ = & 0 & \iff \\
(c_1(L) - \frac{i}{2 \pi} \delta) \cdot \omega_g  = & 0 &    
\end{array}
$$

The last line follows from two facts. First for $v \in \Omega^+(X)$ and $w \in \Omega^-(X)$, we have $v \cdot w = 0$. To see this observe
$$ \begin{array}{r l}
\, v \cdot w  &= \int_X v \wedge w = -\int_X v \wedge \ast w = -\int_X g(v, w) dVol  \\
\, v \cdot w & =  -\int_X  w \wedge  v = -\int_X w \wedge -\ast v = \int_X g(v,  w) dVol.
\end{array}
$$
So $(c_1(L) - \frac{i}{2 \pi} \delta)^+ =0 $ implies
$$(c_1(L) - \frac{i}{2 \pi} \delta) \cdot \omega_g = (c_1(L) - \frac{i}{2 \pi} \delta)^- \cdot \omega_g + (c_1(L) - \frac{i}{2 \pi} \delta)^+ \cdot \omega_g = 0 + 0 = 0.$$
Next given $v, w \in \Omega^+_g(X)$,
$$v \cdot w = c_1 \omega_g \cdot c_2 \omega_g = c_1 c_2 \omega_g^2 = c_1 c_2.$$
This is true since $\text{dim } H^+(X) = b_2^+ = 1$. Hence we see $v \cdot w = 0$ if and only if $[v]= 0$ or $[w]= 0$. Hence $(c_1(L) - \frac{i}{2 \pi} \delta) \cdot \omega_g = 0$ implies  
$
(c_1(L) - \frac{i}{2 \pi} \delta)^- \cdot \omega_g + (c_1(L) - \frac{i}{2 \pi} \delta)^+ \cdot \omega_g = 0$ implies
$0 + (c_1(L) - \frac{i}{2 \pi} \delta)^+ \cdot \omega_g = 0$ implies $(c_1(L) - \frac{i}{2 \pi} \delta)^+ =0 $.

\end{proof}

By the previous lemma, we can divide the space of perturbations $\{ (g, \delta) \}$ for which there are no reducible solutions into a plus chamber and a minus chamber according to the sign of $(c_1(L) - \frac{1}{2 \pi} \delta) \cdot \omega_g$. 

\begin{proposition}
\cite{bplus} The Seiberg-Witten function is well-defined if we only consider perturbations from one of the two chambers. \hfill $\Box$.
\end{proposition} 

Mimicing the definition from the previous section but restricting to perturbations in the plus (or minus) chamber, we obtain the invariants 
$$SW^+_{X, \alpha X} \colon Spin^c(X) \rightarrow \mathbb Z \text{ and } SW^-_{X, \alpha X}\colon Spin^c(X) \rightarrow \mathbb Z. $$

Now suppose $\omega$ is a symplectic form on $X$. 

\begin{lemma} \cite{taubesconst}
$\omega$ determines a homology orientation for $X$.
\hfill $\Box$.
\end{lemma}
We will choose the chambers with respect to this choice (by choosing the $\omega_g$ sharing the same forward cone as $\omega$) and write the Seiberg-Witten functions as $SW^\pm_\omega$. We refer to the minus chamber of perturbations as {\it Taubes' chamber}.

\begin{theorem}\emph{(Taubes' Constraints)}\label{tconst}
\cite{taubesconst} Let $(X,\omega)$ be a closed symplectic $4$-manifold with $b_2^+(X) = 1$, canonical class $K$, and canonical $Spin^c(4)$-structure $\xi_{-K}$.  Then
$$SW^-_\omega(\xi_{K^{-1}}) =  1.$$
If $SW^-_\omega(\xi_{K^{-1}} \otimes E) \neq 0$ then $E \cdot [\omega] \geq 0$ with equality only when $E = 0$.
\hfill $\Box$
\end{theorem}

Now we will prove the following corollary.

\begin{corollary} \label{tcol}
If $\xi \in Spin^c(X)$ satisfies $c_1(L_\xi) \cdot [\omega] = -K \cdot [\omega]$ then $\xi = \xi_{K^{-1}}$.
\end{corollary}

\begin{proof}
Suppose $\xi_{K^{-1}} \otimes E \in Spin^c(X)$ has $SW^-_\omega(\xi = \xi_{K^{-1}} \otimes E) \neq 0$  and satisfies $c_1(L_\xi) \cdot [\omega] = -K \cdot [\omega]$. Then $(-K + 2E) \cdot [\omega] = -K \cdot [\omega]$. It follows $2E \cdot [\omega] = 0$ hence $E \cdot [\omega] = 0$. From Taubes' constraints, $E = 0$. 
\end{proof}

\begin{lemma}\label{sym}\emph{(Symmetry Lemma)}
\cite{taubesconst} Let $(X,\omega)$ be a closed symplectic $4$-manifold with $b_2^+(X) = 1$, canonical class $K$, and canonical $Spin^c(4)$-structure $\xi_{-K}$. Then
$$ SW^+_\omega(\xi_{K^{-1}} \otimes E) = (-1)^{1-b_1(X)/2}SW^-_\omega(\xi_{K^{-1}} \otimes (K-E)).$$\hfill $\Box$
\end{lemma}

\end{subsection}
\end{section}

 \chapter{Circle bundles over 3-manifolds} 

In this chapter, we will define twisted Alexander polynomials for $3$-manifolds and show how they are related to Seiberg-Witten invariants. We will also prove some results about circle bundles necessary for the theorem in the next chapter.

\begin{section}{Twisted Alexander polynomials for $3$-manifolds}

Here, we will define the twisted Alexander polynomials for a path connected, compact, orientable $3$-manifold $Y$ with $b_1(Y) \geq 1$ whose boundary (if any) is a union of disjoint tori.

Throughout this section let $H := H_1(Y; \mathbb Z) / Tor \cong \mathbb Z^{b_1(Y)}$. In order to define our polynomial, we need the following data. 
\begin{itemize}
\item A free abelian group $F$.
\item A nontrivial $\Phi \in Hom(H,F)$.
\item A finite group $G$.
\item An epimorphism $\alpha : \pi_1(Y) \rightarrow G$
\end{itemize}

First note $\alpha \times \Phi$ gives an action of $\pi_1(Y)$ on $G \times F$ as follows. 
$$\begin{array}{r l}
\alpha \times \Phi : \pi_1(Y) & \rightarrow Aut(G \times F) \\
h & \mapsto ((g,f) \mapsto (\alpha(h)(g), f + \Phi(h)))
\end{array}
$$
This induces an action of $\mathbb Z[\pi_1(Y)]$ on $\mathbb Z[G \times F]$. 

Now let $\tilde Y$ denote the universal cover of $Y$ and $C_* (\tilde Y)$ denote its cellular chain complex. $C_*(\tilde Y)$ is a left $\pi_1(Y)$-module via deck transformations. Observe we can extend this action linearly to obtain an action of $\mathbb Z[\pi_1(Y)]$ on $C_*(\tilde Y)$. 

We can also define a right action of $\mathbb Z[\pi_1(Y)]$ on $C_*(\tilde Y)$ by $\sigma \cdot h = h^{-1} \cdot \sigma$ hence we can form the chain complex
$$ C_*(\tilde Y) \otimes_{\mathbb Z[\pi_1(Y)]} \mathbb Z[G \times F].$$
Note that $\mathbb Z [G \times F]$ has a right action by $\mathbb Z[F]$ induced by the right multiplication of $\mathbb Z[F]$. This action and the left action of $\mathbb Z[\pi_1(Y)]$ are associative hence $\mathbb Z [G \times F]$ is a $\mathbb Z[\pi_1(Y)]-\mathbb Z[F]$ bimodule. Therefore the above chain complex has a right action by $\mathbb Z[F]$.

Now consider 
$$H_1^{\alpha \otimes \Phi} (Y; \mathbb Z[G \times F]) := H_1 (C_*(\tilde Y) \otimes_{\mathbb Z[\pi_1(Y)]} \mathbb Z[G \times F]))$$
This inherits the structure of a right $\mathbb Z[F]$-module and is called the {\it twisted Alexander module} of $(Y, \Phi, \alpha)$. Note if we endow $Y$ with a finite cell structure, our module becomes finitely generated. Also our module is finitely related since $\mathbb Z[F]$ is Noetherian.

The {\it twisted Alexander polynomial} of $(Y, \alpha, \Phi)$ is defined to be 
$$\Delta^\alpha_{Y, \Phi} := order \,\, H_1^{\alpha \otimes \Phi} (Y; \mathbb Z[G \times F]) \in \mathbb Z[F] / \pm F$$

\begin{itemize}
\item If $F = \mathbb Z$ then $\Phi \in Hom(H, \mathbb Z) = H^1(Y; \mathbb Z)$ and $\mathbb Z [ \mathbb Z] = \mathbb Z [t^{\pm 1}]$ (the Laurent polynomial ring over $\mathbb Z$) so $\Delta^\alpha_{Y, \Phi} \in \mathbb Z[t^{\pm 1}] / \{ \pm t^n \mid n \in \mathbb Z \}$. We will call this the {\it twisted single variable Alexander polynomial} of $(Y, \alpha, \Phi)$.

\item If $F = H$ and $\Phi \in Hom(H,H)$ is the identity map then we write $\Delta^\alpha_Y \in \mathbb Z[H] / \pm H$ and call it the {\it twisted multivariable Alexander polynomial} of $(Y, \alpha, \Phi)$.

\item If $G$ is the trivial group then we get the {\it ordinary Alexander polynomial} of $(Y, \Phi)$  which we will write $\Delta_{Y, \Phi}$.

\end{itemize}  

Observe if $\nu K \subset S^3$ denotes a tubular neighborhood for a knot $K$ in $S^3$ then $\Delta_{S^3 - \nu K}$ will be the classical Alexander polynomial of the knot $K$.

\begin{proposition} \label{tur} \cite{tur} For each twisted single variable polynomial $\Delta^\alpha_{Y, \Phi} \in \mathbb Z[t^{\pm 1}] / \{\pm t^n \mid n \in \mathbb Z \}$ there is a unique (up to sign) symmetric representative $\Delta^\alpha_{Y, \Phi} \in \mathbb Z[t^{\pm 1}]$. \hfill $\Box$
\end{proposition}

We define the {\it Laurent degree} of a Laurent polynomial to be the difference between its highest and lowest exponents, or $-\infty$ for the zero polynomial. When we are discussing Laurent polynomials, we will refer to Laurent degree as simply degree. Note given $f, g \in \mathbb Z[t^{\pm 1}]$, we have $\text{deg}(f \cdot g) = \text{deg}(f) \cdot \text{deg}(g)$ where $-\infty \cdot n = - \infty$ for each possible degree $n$. In particular, multiplication by powers of $t$ and $\pm 1$ doesn't change the Laurent degree, so it is well-defined for twisted single variable Alexander polynomials.  

\begin{lemma} \label{twistedtocover} \cite{twisted}
$$ \Delta_{Y, \Phi}^\alpha = \Delta_{Y_\alpha, \Phi_\alpha} $$
where $\pi_\alpha : Y_\alpha \rightarrow Y$ is the $G$-cover of $Y$ induced by $\alpha$ and $\Phi_\alpha = \pi^*_\alpha(\Phi)$.
\end{lemma}

\end{section}

\begin{section}{The Alexander and Thurston norms}

\begin{definition}
If $Y$ is a connected, closed, oriented $3$-manifold and $\Phi \in H^1(Y; \mathbb Z)$, we define the {\it Alexander norm} of $\Phi$ to be 
$$\| \Phi \|_A := \text{sup}_{i,j} \Phi (h_i - h_j)$$
where $\sum a_i \cdot h_i  \in \mathbb Z[H]$ represents $\Delta_Y$ (we assume the $a_i$ are nonzero and the $h_i$ are nonzero and distinct). 
By convention $\| \cdot \|_A \equiv 0$ if $\Delta_X = 0$.
\end{definition}

\begin{definition}
For a $3$-manifold $Y$, we say a class $\phi \in H_2(Y; \mathbb Z)$ is {\it represented by} an oriented closed surface $\Sigma$ if there is an embedding $i \colon \Sigma \hookrightarrow Y$ such that $i_*([\Sigma]) = \phi$.
\end{definition}

\begin{definition}
If $Y$ is a $3$-manifold and $\phi \in H_2(Y; \mathbb Z)$, we define the {\it Thurston norm} of $\phi$ to be 
$$\| \phi \|_T := \text{min} \{\chi\_ (\Sigma) \mid \Sigma \text{ is an embedded oriented closed surface representing } \phi \}.$$ 
\end{definition}
Here, given a surface $S$ with connected components $S_1 \sqcup \cdots \sqcup S_k$, we define 
$$\chi\_ (S) = \sum_{i=1}^k \text{max} \{-\chi(S_i), 0 \}.$$
Observe the Thurston norm is always nonnegative and even valued as for each connected component $S_i$ of $S$, we have $\chi (S_i) = 2 - 2 \text{ genus}(S_i)$. 

Next we will prove the following standard result.

\begin{lemma} \label{thurstk}
The Thurston norm satisfies $\|k\Phi \|_T = | k | \|\Phi\|_T$ for each $k \in \mathbb Z$.
\end{lemma}

\begin{proof}

Let $S$ be a surface representing $\Phi$ with minimal complexity (ie with minimal genus). $k$ parallel copies of $\Phi$ represent $k \Phi$ so $\| k \|_T \leq k \| \Phi \|_T$.

Now let $S$ be a surface representing $k \Phi$ with minimal complexity. By the Pontrajin construction, there exists a map $f_{k \Phi} : Y \rightarrow S^1$ induced by $k \Phi$ with a regular value $y \in S^1$ where $f_{k \Phi}^{-1}(y) = S$. Let $f_\Phi$ denote the map induced by $\Phi$. We have the following diagram that commutes up to homotopy.

$$
\begin{tikzpicture}[description/.style={fill=white,inner sep=2pt}]
\matrix (m) [matrix of math nodes, row sep=3em,
column sep=2.5em, text height=1.5ex, text depth=0.25ex]
{   Y  & S^1  \\
       &  S^1 \\
};
\path[->,font=\scriptsize]
(m-1-1) edge node[auto]{$f_\Phi$} (m-1-2);
\path[->,font=\scriptsize]
(m-1-1) edge node[auto, swap]{$f_{k \Phi}$} (m-2-2);
\path[->,font=\scriptsize]
(m-1-2) edge node[auto]{$\pi$} (m-2-2);
\end{tikzpicture}
$$
where $\pi$ denotes the covering map of degree $k$.

By the covering homotopy property, we can homotope $f_\Phi$ to make the diagram commute. Then $S = f_{k \Phi}^{-1}(y) = f_{\Phi}^{-1}(y_1) \cdots f_{\Phi}^{-1}(y_k)$ where $y_1, \cdots , y_k$ are the preimages of $y$ under $\pi$. Since each $f_{\Phi}^{-1}(y_i)$ represents $\Phi$, we have $k \| \Phi \|_T \leq \| k \Phi \|_T$.
\end{proof}

\begin{definition}
For a homomorphism $\Phi \in \text{Hom}(H,\mathbb Z)$, we define the divisibility of $\Phi$ as 
$$\text{div}(\Phi) = \text{max} \{k \mid \Phi = k \Phi^\prime \text{ for some } \Phi^\prime \in \text{Hom}(H,\mathbb Z) \}.$$  
\end{definition}

\begin{theorem} (McMullen's Inequality) \label{mcmullen} \cite{mcmullen} If $Y$ is a connected, closed, oriented $3$-manifold with $b_1(Y) \neq 0$ then for each $\Phi \in H^1(Y; \mathbb Z) \cong Hom(H_1(Y), \mathbb Z) = Hom(H, \mathbb Z)$ we have
$$ \| \Phi \|_A \leq \| \text{PD}(\Phi) \|_T + 
  \left\{
  \begin{array}{l l}
    0 & \quad \text{if $b_1(Y) > 1$,}\\
    2div\Phi & \quad \text{if $b_1(Y) = 1$.}\\
  \end{array} \right.$$ \hfill $\Box$  
\end{theorem}

\begin{proposition} \label{degprop} \cite{twisted}
If $b_1(Y) > 1$ then for each $\Phi \in H^1(Y; \mathbb Z)$, 
$$\text{deg} \Delta_{Y, \Phi} \leq \| \Phi \|_A + 2\text{ div} (\Phi).$$ \hfill $\Box$
\end{proposition}

Observe each $\Phi \in H^1(Y; \mathbb Z) = Hom(H, \mathbb Z)$ induces a ring homomorphism $$\Phi : \mathbb Z[H] / \pm H \rightarrow \mathbb Z[t^{\pm 1}] / \{ \pm t^n \mid n \in \mathbb Z \}.$$

\begin{proposition} \label{norm} \cite{twisted}
$$\Delta_{Y, \Phi} =
  \left\{
  \begin{array}{l l}
  (t^{\text{div}(\Phi)}-1)^2 \cdot \Phi(\Delta_Y) & \text{if $b_1(Y) > 1$.} \\
  \Phi(\Delta_Y) & \text{if $b_1(Y) = 1$.}
  \end{array} \right.
$$
\hfill $\Box$
\end{proposition}

\end{section}

\begin{section}{The 3-dimensional Seiberg-Witten invariants}

Similar to the four dimensional case, we can also define Seiberg-Witten invariants for an oriented Riemannian $3$-manifold $Y$. We have the map
$$ SW_{Y, \alpha Y} : Spin^c(Y) \rightarrow \mathbb Z$$
where $Spin^c(Y)$ denotes the set of equivalence classes of $Spin^c(3)$ structures. For a detailed explanation, we refer the reader to \cite{swthree}.

The following theorem relates these invariants with the multivariable Alexander polynomial of $Y$.

\begin{theorem}\label{meng} \emph{(Meng-Taubes)} \cite{meng}
Let $Y$ be a closed oriented $3$-manifold with $b_1(Y) > 1$ and let $H = H_1(Y; \mathbb Z) / Tor$. Then
$$ \Delta_Y =  \pm  \sum_{\zeta \in Spin^c(Y)} SW_{Y, \alpha Y} (\zeta) \cdot \frac{1}{2} f(c_1(L_\zeta)) \in \mathbb Z[H] / \pm H$$
where $f$ denotes the composition of Poincare duality and the quotient map $H_1(Y; \mathbb Z) \rightarrow H$.  \hfill $\Box$
\end{theorem}

$c_1(L_\zeta)$ has even divisibility for each $\zeta \in Spin^c(Y)$ so multiplication by $\frac{1}{2}$ is well-defined. To see this, recall $Y$ is parallelizable so it admits a trivial $Spin^c(3)$-structure hence each $c_1(L_\zeta)$ may be written $0 + 2E$ for some $E \in H^2(Y)$.

\end{section}

\begin{section}{Pullback of $\bm{Spin^c(3)}$-structures}

Throughout this section, we will let $X$ be an oriented, closed, Riemannian \\ $4$-manifold, $Y$ be a closed, oriented, oriented $3$-manifold, and $p : X \rightarrow Y$ be a principal $S^1$-bundle.

Let $\eta \in \wedge^1(Y)$ denote a principal $S^1$-connexion of $p$ and let $g_Y$ be the metric on $Y$. We can endow
$X$ with the metric $g_X = \eta \otimes \eta + p^*(g_Y)$. Using this metric, there is an
orthogonal splitting
$$T^* X = \mathbb R \eta \oplus p^*(T^* Y).$$
Note this splitting is independent of the choice of $\eta$.

If $\zeta = (W, \mu)$ is a $Spin^c(3)$-structure on $Y$, we define the {\it pullback} of $\zeta$ to be
$(p^*(W) \oplus p^*(W), \mu^\prime)$ where $ \mu^\prime : T^* X \rightarrow End_\mathbb C (p^*(W) \oplus p^*(W))$ is given by
$$ \mu^\prime (b \eta + p^*(a)) = 
\left[ {\begin{array}{ll}
 0 & p^*(\mu(a)) + b Id_{p^*(W)} \\
 p^*(\mu(a)) - b Id_{p^*(W)} & 0
\end{array} } \right] .
$$
We will write the pullback map as $p^* : Spin^c(Y) \rightarrow p^*Spin^c(Y) \subset Spin^c(X)$.

\begin{lemma}
\cite{bald2} We have the following relationship between determinant line bundles: $c_1(L_{p^*(\zeta)}) = p^*(c_1(L_\zeta))$ for each $\zeta \in Spin^c(Y)$. \hfill $\Box$
\end{lemma} 

\begin{lemma}
The pullback map $p^* : Spin^c(Y) \rightarrow Spin^c(X)$ is equivariant with respect to the action of $H^2(Y; \mathbb Z)$ on $Spin^c(Y)$ and the action of $H^2(Y; \mathbb Z)$ on $p^*Spin^c(Y)$ via $p^* : H^2(Y; \mathbb Z) \rightarrow H^2(X; \mathbb Z)$. Hence $p^*H^2(Y; \mathbb Z)$ acts freely and transitively on $p^* Spin^c(Y) \subset Spin^c(X)$.
\end{lemma}

\begin{proof}

Choose $\zeta \in Spin^c(Y)$ and let $W$ denote its complex spinor bundle and let $\mu : T^*Y \rightarrow End(W)$ denote its Clifford multiplication map. For $E \in H^2(Y)$, $p^*(\zeta \otimes E)$ has complex spinor bundle
$$ p^*(W \otimes E) \oplus p^*(W \otimes E) = (p^*(W) \oplus p^*(W)) \otimes p^*(E).$$
This is precisely the complex spinor bundle of $p^*(\zeta) \otimes p^*(E)$

Now $p^*(\zeta) \otimes p^*(E)$ has the induced map
$$
\begin{array}{r l}
 &
\left[ {\begin{array}{ll}
 0 & p^*(\mu(a)) + b Id_{p^*(W)} \\
 p^*(\mu(a)) - b Id_{p^*(W)} & 0
\end{array} } \right] \otimes Id_{p^*(E)}  \\  = &
\left[ {\begin{array}{ll}
 0 & (p^*(\mu(a)) + b Id_{p^*(W)}) \otimes Id_{p^*(E)} \\
 (p^*(\mu(a)) - b Id_{p^*(W)}) \otimes Id_{p^*(E)} & 0
\end{array} } \right]
\\  = &
\left[ {\begin{array}{ll}
 0 & p^*(\mu(a) \otimes Id_E) + b Id_{p^*(W \otimes E)} \\
 p^*(\mu(a) \otimes Id_E) - b Id_{p^*(W \otimes E)}) & 0
\end{array} } \right].
\end{array}
$$
which is precisely the Clifford multiplication map for $p^*(\zeta \otimes E)$ so we see $p^*(\zeta \otimes E) = p^*(\zeta) \otimes p^*(E)$.
\end{proof}

An easy consequence of the Gysin sequence is the following lemma.

\begin{lemma}\label{b+} \cite{vanish}
If $p: X \rightarrow Y$ has a nontorsion Euler class, $e(X) \in H^2(Y; \mathbb Z)$ then $b_2^+(X) = b_1(Y)-1$. \hfill $\Box$
\end{lemma}

\begin{theorem}\label{bald} \cite{bald2} \emph{(Baldridge)} Suppose $p : X \rightarrow Y$ has a nontorsion Euler class and $Y$ has $b_1(Y) = 2$ (note $b_2^+(X) = 1$ from the previous lemma).
If $\xi \in Spin^c(X)$ is pulled back from a $Spin^c(3)$-structure on $Y$ then
$$ SW_X^{\pm}(\xi) = \sum_{\zeta \in (p^*)^{-1}(\xi)} SW_Y^3(\zeta).  $$
\hfill $\Box$
\end{theorem}

\begin{remark}
Baldridge's designation of the two chambers into {\it plus} and {\it minus} is different from the approach taken in this text but as long as we only consider pulled back structures, this will not affect us.
\end{remark}

\begin{proposition}\cite{vanish}\label{kpullback} If $X$ admits a symplectic structure $\omega$ and $p : X \rightarrow Y$ has a nontorsion Euler class then the canonical class $K$ is contained in the image of $p^*:H^2(Y; \mathbb Z) \rightarrow H^2(X; \mathbb Z)$. \hfill $\Box$
\end{proposition}

\begin{lemma}
Suppose $p : X \rightarrow Y$ has a nontorsion Euler class. If $\xi \in Spin^c(X)$ has $c_1(L_\xi) \in p^*H^2(Y;\mathbb Z)$ then $\xi \in p^*Spin^c(Y)$. In particular, if $X$ admits a symplectic structure $\omega$ then the canonical $Spin^c(4)$-structure $\xi_{K^{-1}}$ is contained in  $p^*Spin^c(Y)$.
\end{lemma}

\begin{proof}

Let $\xi \in Spin^c(X)$ with $c_1(L_\xi) \in p^*H^2(Y; \mathbb Z)$. First recall all oriented $3$-manifolds are parallelizable so we have a trivial $Spin^c(3)$-structure $\Xi_Y$. From the free and transitive action of $H^2(X; \mathbb Z)$, we can write $\xi = p^*(\Xi_Y) \otimes E$ where $2E = c_1(\xi)$ hence $c_1(\xi)$ has even divisibility. Now the Gysin sequence descends to one with coefficients in $\mathbb Z_2$.
$$
\begin{tikzpicture}[description/.style={fill=white,inner sep=2pt}]
\matrix (m) [matrix of math nodes, row sep=3em,
column sep=2.5em, text height=1.5ex, text depth=0.25ex]
{ <e(X)>     &  H^2(Y; \mathbb Z) & H^2(X; \mathbb Z)  \\
  <\pi(e(X))> & H^2(Y; \mathbb Z_2) & H^2(X; \mathbb Z_2)  \\
};
\path[->,font=\scriptsize]
(m-1-2) edge node[auto]{$p^*$} (m-1-3);
\path[->,font=\scriptsize]
(m-2-2) edge node[auto]{$p^*$} (m-2-3);

\path[->,font=\scriptsize]
(m-1-1) edge node[auto]{} (m-1-2);
\path[->,font=\scriptsize]
(m-2-1) edge node[auto]{} (m-2-2);

\path[->,font=\scriptsize]
(m-1-3) edge node[auto]{$\pi$} (m-2-3);
\path[->,font=\scriptsize]
(m-1-1) edge node[auto]{$\pi$} (m-2-1);
\path[->,font=\scriptsize]
(m-1-2) edge node[auto]{$\pi$} (m-2-2);
\end{tikzpicture}
$$
where $\pi$ denotes the corresponding quotient maps.

Choose a $B \in (p^*)^{-1}(c_1(L_\xi))$. Since $c_1(L_\xi)$ has even divisibility, $\pi(c_1(L_\xi)) = 0$ hence from the diagram, $p^*(\pi(B)) = 0$. Thus we have $B \equiv n \, e(X) \text{ (mod 2)}$ for some $n \in \mathbb Z$ so $B = n \, e(X) + F$ for some $F \in H^2(Y; \mathbb Z)$ with even divisibility. Now choose a $D \in H^2(Y; \mathbb Z)$ with $2D = F$. Note $E = p^*(B) = p^*(n \, e(X)) + p^*(F) = p^*(F)$. Using the invariance of the pullback map, $p^*(\Xi_Y \otimes D) = p^*(\Xi_Y)  \otimes p^*(D)$ and this pulled back $Spin^c(4)$-structure's determinant line bundle $L$ will have $c_1(L) = 2p^*(D) = p^*(2D) = p^*(F) = E$.

Using the free and transitive action of $H^2(X; \mathbb Z)$ again, we have $\xi = p^*(\Xi_Y \otimes D) \otimes A$ for some $2$-torsion $A \in H^2(X; \mathbb Z)$. But, it follows from the Gysin sequence
$$ \cdots \rightarrow H^2(Y; \mathbb Z)  \xrightarrow{p^*} H^2(X; \mathbb Z) \xrightarrow{p_*} H^1(Y; \mathbb Z) \rightarrow \cdots $$
that all the torsion elements are contained in $p^*H^2(Y; \mathbb Z)$. Therefore since $p^*H^2(Y; \mathbb Z)$ acts freely and transitively on $p^*Spin^c(Y)$, the result follows.
\end{proof}

\begin{lemma}
Suppose $p : X \rightarrow Y$ has a nontorsion Euler class, $Y$ has $b_1(Y) = 2$ (note $b_2^+(X) = 1$), and $X$ admits a symplectic structure $\omega$ with canonical class $K$. Then $SW^-(\xi_{K^{-1}} \otimes K) \neq 0$. If $\xi \in p^*Spin^c(Y)$ satisfies $SW^-(\xi) \neq 0$ and $c_1(L_\xi) \cdot [\omega] = K \cdot [\omega]$ then $\xi = \xi_{K^{-1}} \otimes K$.
\end{lemma}

\begin{proof}

From the Symmetry Lemma (Lemma \ref{sym}) and Taubes' constraints (Theorem \ref{tconst}), we have $SW^+(\xi_{K^{-1}} \otimes K) \neq 0$. Observe $\xi_{K^{-1}} \otimes K \in p^*H^2(Y; \mathbb Z)$ from the free and transitive action of $p^*H^2(Y; \mathbb Z)$ on $p^*Spin^c(Y)$.

From Baldridge (Theorem \ref{bald}), we have $SW^+(\xi)= SW^-(\xi)$ for $\xi \in p^*Spin^c(Y)$ so $SW^-(\xi_{K^{-1}} \otimes K) = SW^+(\xi_{K^{-1}} \otimes K) \neq 0$.

Now suppose $\xi = \xi_{K^{-1}} \otimes E \in p^*Spin^c(X)$ with $E \in p^*H^2(Y, \mathbb Z)$ satisfies $SW^-(\xi) \neq 0$ and $c_1(L_\xi) \cdot [\omega] = K \cdot [\omega]$. Then $(-K+2E) \cdot [\omega] = K \cdot [\omega]$ implies $-2(K-E) \cdot [\omega] = 0$ hence $E \cdot [\omega] = K \cdot [\omega]$. Now from the free and transitive action of $p^*H^2(Y; \mathbb Z)$, $\xi_{K^{-1}} \otimes (K-E) \in p^*Spin^c(Y)$. By the Symmetry Lemma and Baldridge, $SW^-(\xi_{K^{-1}} \otimes (K-E)) \neq 0$. But $(K-E) \cdot [\omega] = K \cdot [\omega] - E \cdot [\omega] = 0$ so by Taubes' constraints, $K-E = 0$ thus $E = K$.
\end{proof}

\begin{lemma} \label{baldcor}
Suppose $p : X \rightarrow Y$ has a nontorsion Euler class $e(X) \in H^2(Y; \mathbb Z)$ and $Y$ has $b_1(Y) = 2$ (note $b_2^+(X) = 1$). Then for each $\Phi \in H^1(Y; \mathbb Z)$ and $\sigma \in H^2(X, \mathbb Z)$ satisfying $\Phi = p_*(\sigma)$, we have
$$ \Phi(\Delta_Y) = \pm \sum_{\xi \in p^* Spin^c(Y)}  SW_X^-(\xi) t^{\frac{1}{2} c_1(L_\xi) \cdot \sigma}.$$
\end{lemma}
\begin{proof}

Recall by Meng and Taubes's Theorem \ref{meng}, we have
$$ \Delta_Y = \pm \sum_{\zeta \in Spin^c(Y)} SW_Y(\zeta) \cdot \frac{1}{2} f(c_1(L_\zeta)) \in \mathbb Z[H].$$
Note $\Phi(f(c_1(L_\zeta))) = c_1(L_\zeta) \cdot \Phi$ so 
$$\Phi(\Delta_Y) = \pm \sum_{\zeta \in Spin^c(Y)} SW_Y(\zeta) t^{\frac{1}{2} c_1(L_\zeta) \cdot \Phi}.$$
Now observe for each $\zeta \in Spin^c(Y)$, we have 
$$c_1(L_\zeta) \cdot \Phi = c_1(L_\zeta) \cdot p_*(\sigma) = p^*(c_1(L_\zeta)) \cdot \sigma  = c_1(L_\xi) \cdot \sigma $$
where $\xi = p^* (\zeta) \in p^*Spin^c(Y) \subset Spin^c(X)$. Grouping the terms in the sum by their induced pullback $Spin^c(4)$-structures and using a result of Baldridge (Theorem \ref{bald}), we obtain
$$ \begin{array}{l l}
     \Phi(\Delta_Y) &= \pm \sum_{\xi \in p^* Spin^c(Y)} \sum_{\zeta \in (p^*)^{-1}(\xi)} SW_Y(\zeta) t^{\frac{1}{2} c_1(L_\xi) \cdot \sigma} \\  & \\
                       &= \pm \sum_{\xi \in p^* Spin^c(Y)}  SW_X(\xi)^- t^{\frac{1}{2} c_1(L_\xi) \cdot \sigma}.
   \end{array}
$$
\end{proof}

\begin{corollary} If $X$ admits a symplectic form $\omega$ and $p_*([\omega]) = \Phi$ for some $\Phi \in H^1(Y; \mathbb Z)$ then $\text{deg}(\Phi(\Delta_{Y})) = K \cdot [\omega]$ where $K \in p^*H(Y; \mathbb Z)$ is the canonical class of $\omega$.
\end{corollary}

\begin{proof}

We will show for each $\xi \in p^*Spin^c(Y)$ with $SW^-(\xi) \neq 0$, $-K \cdot [\omega] \leq c_1(L_\xi) \cdot [\omega] \leq K \cdot [\omega]$. 

Using the free and transitive action of $p^*H(Y; \mathbb Z)$ on $p^*Spin^c(Y)$, we can write $\xi = \xi_{K^{-1}} \otimes E$ for some $E \in p^*H(Y; \mathbb Z)$ where $\xi_{K^{-1}} \in p^*Spin^c(Y)$ denotes the canonical $Spin^c(4)$-structure of $\omega$.

First suppose $c_1(L_\xi) \cdot [\omega] \leq -K \cdot [\omega] $. This implies $(-K+2E) \cdot [\omega] \leq -K \cdot [\omega]$ hence $E \cdot [\omega] \leq 0$. Therefore by Taubes' constraints, $E = 0$. So we have $c_1(L_\xi) \cdot [\omega] = -K \cdot [\omega]$.

Now suppose $c_1(L_\xi) \cdot [\omega] \geq K \cdot [\omega]$. Hence $(-K+2E) \cdot [\omega] \geq K \cdot [\omega]$ and it follows $0 \geq (K-E) \cdot [\omega]$. By the symmetry lemma, $SW^-(\xi_{K^{-1}} \otimes (K-E)) \neq 0$ so we have $E = K$ by Taubes' constraints. Therefore $c_1(L_\xi) \cdot [\omega]= K \cdot [\omega]$.

Recall in this section, we showed $SW^-(\xi_{K^{-1}}) \neq 0$ and $SW^-(\xi_{K^{-1}} \otimes K) \neq 0$ so the desired result follows from the previous lemma. 
\end{proof}

\end{section}

\begin{section}{Fibred Classes}

\begin{definition}
Given a manifold $Y$, we say $\Phi \in H^1(Y; \mathbb Z)$ {\it fibres over} $S^1$ if the homotopy class of maps $\Phi \in H^1(Y; \mathbb Z) = [Y; K(\mathbb Z, 1)] = [Y; S^1]$ contains a representative that makes $Y$ into a fibre bundle over $S^1$; in this case we will call $\Phi$ a {\it fibred class}.
\end{definition}

\begin{lemma}\label{fibredk} For a manifold $Y$, $\Phi \in H^1(Y; \mathbb Z)$ is a fibred class if and only if $k\Phi$ is also a fibred class for each $k \in \mathbb Z$. \hfill $\Box$
\end{lemma}

\begin{theorem} (Friedl-Vidussi) \cite{lubotzsky} Let $(X, \omega)$ be a closed symplectic manifold with trivial canonical class. If
$X \rightarrow Y$ is a principal $S^1$-bundle with nontorsion Euler class $e(X) \in H^2 (Y ; \mathbb Z)$, then $Y$ is a torus bundle over a circle. \hfill $\Box$
\end{theorem}

\begin{corollary}\label{vcol} Under the hypotheses of the previous theorem, the Thurston norm of $Y$ vanishes everywhere.
\end{corollary}

\begin{proof}
Let $T^2 \hookrightarrow Y \rightarrow S^1$ denote a torus bundle over a circle. This may be realized as the mapping torus for a diffeomorphism $f : T^2 \rightarrow T^2$ with fixed point $x_0$. We now see from our CW-decomposition in 6.1 that $H_2(Y)$ can be generated by tori so the Thurston norm of $Y$ must vanish everywhere.
\end{proof}

\begin{theorem} (Friedl-Vidussi) \label{detect} \cite{detect} Let $Y$ be a closed oriented connected $3$-manifold.
Let $\Phi \in H^1 (Y ; \mathbb Z)$ be a nontrivial class. If for any homomorphism $\alpha \colon \pi_1 (Y ) \rightarrow G$ to a finite group, the twisted Alexander polynomial $\Delta_{Y,\Phi}^\alpha \in \mathbb Z[t^{\pm 1}]$ is monic and
$$\text{deg}(\Delta_{Y,\Phi}^\alpha ) = |G| \|\Phi \|_T + (1 + b_3 (Y))div\Phi_\alpha$$
holds, then $\Phi$ is a fibred class. \hfill $\Box$
\end{theorem}

\end{section}

 \chapter{A partial converse to Fern\'andez-Gray-Morgan's \\ theorem}

\begin{remark}
Suppose we are given an oriented, closed, symplectic $4$-manifold $(X, \omega)$ and a closed, oriented, connected $3$-manifold $Y$ together with a principal $S^1$-bundle $p: X \rightarrow Y$ with nontorsion Euler class. Using the openness of the symplectic condition, we can assume that $[\omega] \in H^2(X ; \mathbb R)$ lies in the rational lattice, which we will identify with $H^2(X ; \mathbb Q)$. After suitably scaling $\omega$ by a rational number if needed and using Poincare duality, the class $p_*[\omega]$ will be a primitive nonzero class in $H^1(Y ; \mathbb Z)$. We will assume all our symplectic structures to be of this form.
\end{remark}

Fern\'andez-Gray-Morgan's theorem is given as follows. Note this gives many examples of non-K\"ahler symplectic $4$-manifolds.

\begin{theorem}\emph{(Fern\' andez-Gray-Morgan)} Let $Y$ be a closed, oriented, connected $3$-manifold. Suppose $p : X \rightarrow Y$ is a principal $S^1$-bundle with Euler class $e(X) \in H^2(Y; \mathbb Z)$.  If there exists a nonzero fibred class $\Phi \in H^1(Y; \mathbb Z)$ satisfying $e(X) \cdot \Phi = 0$ then $X$ is an oriented closed $4$-manifold that admits a symplectic structure. \hfill $\Box$
\end{theorem}

We will prove the following partial converse.

\begin{theorem}
Let $(X, \omega)$ be an oriented, closed, symplectic $4$-manifold and $Y$ be a closed, oriented, connected $3$-manifold with $b_1(Y) = 2$. Suppose $p : X \rightarrow Y$ is a principal $S^1$-bundle with Euler class $e(X) \in H^2 (Y; \mathbb Z)$. 
If $e(X)$ is nontorsion, $\Phi \in H^1(Y; \mathbb Z)$ satisfies $e(X) \cdot \Phi = 0$ and $\|\Phi \|_T = 2$ then $\Phi$ is a fibred class.
\end{theorem}

\begin{proof}

Consider $p_*[w] \in H^1(Y; \mathbb Z)$. First we'll show that up to sign, $\Phi$ equals $p_*[\omega]$.
From the Gysin sequence, $e(X) \cdot p_*[\omega] = 0$. Since $p_*[\omega]$ is primitive and the subspace $<e(X), \cdot> = 0$ is one-dimensional, we have $p_*[\omega]$ must be a generator of $<e(X), \cdot> = 0$.
Now we'll show $\Phi$ must also be primitive. Suppose $\Phi = k \Phi^\prime$ for some primitive $\Phi^\prime$ then by Lemma \ref{thurstk}, $$\|\Phi\|_T = \|k\Phi^\prime \|_T = k \| \Phi^\prime \|_T = 2.$$ But the Thurston norm is always even so $k = \pm 1$ and $\Phi$ must be primitive.
So we may write $\Phi = \pm p_*[\omega]$. Note $\|p_*[\omega] \|_T = 2$. By Lemma \ref{fibredk}, it is sufficent to show $p_*[\omega]$ is fibred. Henceforth we will denote $p_*[\omega]$ by $\Phi$.

From Proposition \ref{degprop}, Mcmullen's inequality (Theorem \ref{mcmullen}), and the fact $\Phi$ is primitive, we obtain 
$$\text{deg} \Delta_{Y, \Phi} \leq \| \Phi \|_A + 2 \text{ div} (\Phi) \leq \| \Phi \|_T+2 \leq 4.$$ 

Since $\Delta_{Y, \Phi}$ has a symmetric representative (Proposition \ref{tur}), its degree is either even and nonnegative or $-\infty$ so $\text{deg}(\Delta_{Y, \Phi}) = -\infty$, $0$, $2$, or $4$. We will examine the consequences of each case. \\

\noindent
{\bf Case 1: $\bm{\text{deg}(\Delta_{Y, \Phi}) = -\infty}$}
\newline

This means $\Delta_{Y, \Phi} = 0$. Since $b_2^+(X) = 1$ (Lemma \ref{b+}), we can use Proposition \ref{norm} and Lemma \ref{baldcor} to obtain  
$$ \sum_{\xi \in p^* Spin^c(Y)}  SW^-_X(\xi) t^{\frac{1}{2} c_1(L_\xi) \cdot [\omega]} = 0.$$

Now we will apply Taubes' constraints (Theorem \ref{tconst}). Let $\xi_{K^{-1}} \in Spin^c(X)$ denote the canonical $Spin^c(4)$-structure of $(X, \omega)$, Note this is a pullback class from Proposition \ref{kpullback}. From Corollary \ref{tcol}, $\xi_{K^{-1}}$ is the only basic class whose determinant line bundle $L$ satisfies $c_1(L) \cdot [\omega] = -K \cdot [\omega]$. Thus the term $SW^-_X(\xi_{K^{-1}}) t^{\frac{1}{2} K^{-1} \cdot [\omega]}$  cannot be killed off by another term as its power is unique so we reach a contradiction. \\

\noindent
{\bf Case 2: $\bm{\text{deg}(\Delta_{Y, \Phi}) = 0}$}
\newline

Since $\Delta_{Y,\Phi} = (t-1)^2 \cdot \Phi(\Delta_Y)$, we have $\text{deg}(\Phi(\Delta_Y)) = -2$ but this is impossible. \\

\noindent
{\bf Case 3: $\bm{\text{deg}(\Delta_{Y, \Phi}) = 2}$}
\newline

Here $\text{deg}(\Phi(\Delta_Y)) = 0$. By Corollary \ref{kpullback}, $\xi_{K^{-1}} \otimes K$ is the only pulled back basic class whose determinant line bundle $L$ satisfies $c_1(L) \cdot [\omega] = K \cdot [\omega]$. So $K = 0$ or the polynomial
$$ \Phi(\Delta_Y) = \pm \sum_{\xi \in p^* Spin^c(Y)}  SW^-_X(\xi) t^{\frac{1}{2} c_1(L_\xi) \cdot [\omega]}$$
(when simplified) would have at least two nonzero terms hence not have zero degree.

Since $K = 0$, as a result of a theorem of Friedl and Vidussi (Corollary \ref{vcol}), the Thurston norm of $Y$ must vanish everywhere. But this contradicts our hypotheses. \\

\noindent
{\bf Case 4: $\bm{\text{deg}(\Delta_{Y, \Phi}) = 4}$}
\newline

Let $\alpha : \pi_1(Y) \rightarrow G$ be an onto homomorphism for some finite group $G$. $\alpha$ induces $G$-covers of $X$ and $Y$ that fit into the following commutative diagram.
$$
\begin{tikzpicture}[description/.style={fill=white,inner sep=2pt}]
\matrix (m) [matrix of math nodes, row sep=3em,
column sep=2.5em, text height=1.5ex, text depth=0.25ex]
{ X_\alpha     &  Y_\alpha  \\
  X & Y  \\
};
\path[->,font=\scriptsize]
(m-1-1) edge node[auto]{$p^\prime$} (m-1-2);
\path[->,font=\scriptsize]
(m-2-1) edge node[auto]{$p$} (m-2-2);

\path[->,font=\scriptsize]
(m-1-1) edge node[auto]{$\pi_\alpha$} (m-2-1);
\path[->,font=\scriptsize]
(m-1-2) edge node[auto]{$\pi_\alpha$} (m-2-2);
\end{tikzpicture}
$$

Slightly abusing notation, we denote both covering maps by $\pi_\alpha$. $\omega_\alpha = \pi_\alpha^*(\omega)$ will be a symplectic form on $X_\alpha$ with canonical class $K_\alpha = \pi^*_\alpha(\omega)$ and canonical $Spin^c(4)$-structure $\xi_{K^{-1}_\alpha}$. Note $X_\alpha$ is closed with $b_2^+(X_\alpha) \geq b_2^+(X) = 1$ and $Y_\alpha$ is connected with $b_1(Y_\alpha) \geq b_1(Y) = 2$.

By Lemma \ref{twistedtocover}, we have $\Delta_{Y, \Phi}^\alpha = \Delta_{Y_\alpha, \Phi_\alpha}$ where $\Phi_\alpha = \pi_\alpha^*(\Phi)$. Hence by Lemma \ref{norm},
$$ \Delta_{Y_\alpha, \Phi_\alpha} = \pm (t^{\text{div} \Phi_\alpha} + 1)^2 \sum_{\xi \in (p^\prime)^* Spin^c(Y_\alpha)}  SW^-_{X_\alpha}(\xi) t^{\frac{1}{2} c_1(L_\xi) \cdot [\omega_\alpha]}.$$ 

By Lemma \ref{degprop}, $\text{deg}(\Delta_{Y_\alpha, \Phi_\alpha}) = 2 \text{div}(\Phi_\alpha) + K_\alpha \cdot [\omega_\alpha]$. But $$K_\alpha \cdot [\omega_\alpha] = |G| K \cdot [\omega] = |G| \text{deg}(\Phi(\Delta_Y) = 2 |G| = |G| \|\Phi\|_T.$$

Note $\Delta_{Y_\alpha, \Phi_\alpha}$ is monic as $SW^-_{X_\alpha}(\xi_{K^{-1}_\alpha} \otimes K_\alpha) = \pm 1$ and $\xi_{K^{-1}_\alpha} \otimes K_\alpha$ is the only pulled back $Spin^c(4)$-structure with $c_1(L_{\xi_{K^{-1}_\alpha} \otimes K_\alpha}) \cdot [\omega_\alpha] = K_\alpha \cdot [\omega_\alpha]$ AND whose evaluation in $SW^-_{X_\alpha}$ is nonzero. Therefore by Theorem \ref{detect}, $\Phi$ must be a fibred class.
\end{proof}

 \chapter{Surface Bundles over Tori}

The purpose of this chapter is to work towards deciding whether surface bundles over a surface are spin. Ron Stern originally raised the question, ``is there an orientable aspherical surface bundle over the torus that is not spin'' \cite{hillman}. In the case of torus bundles over the torus, the total spaces are indeed spin manifolds. Our strategy will be to use the canonical class from Thurston's symplectic form construction to show the manifolds given below have an even intersection forms. This together with the fact that they have no $2$-torsion in their first homology groups implies they are spin.

Now we will give a simple example of a family of surface bundles whose total spaces are spin. Let $\Sigma_h \hookrightarrow M \rightarrow \Sigma_g$ be a surface bundle with genus $h$ fibre over a genus $g$ surface. If $M$ admits a symplectic structure $\omega$ with canonical class $K = 0$ then the $M$ will be spin. To see this, recall $-K$ is the determinant line bundle of the canonical $Spin^c$-structure of $\omega$ so $0 \equiv w_2(M) \text{ (mod 2)}$ and $M$ is spin.

\begin{section}{The mapping torus of a Dehn twist}

First we will list the fundamental group and a CW-decomposition for the mapping torus for a basepoint preserving map $f : T^2 \rightarrow T^2$.

Consider a torus $T^2$ with $\pi_1(T^2) = <a,b \mid >$ where $a$ and $b$ are the usual generators. 

Now let $N$ denote the mapping torus for $f$. We know
$$ \begin{array}{r l}
\pi_1(N) &=  <a,b,c \mid [a,b \,], acf(a)^{-1}c^{-1},  bc f(b)^{-1} c^{-1}>. \\ 
   \end{array}
$$

Now we will define a CW-structure for $N$. We have:
$$\begin{array}{l l}
\text{1 0-cell} & e^0 \\
\text{3 1-cells} & a, b, c \\
\text{3 2-cells} & e_1^2, e^2_2, e^2_3 \\
\text{1 3-cell} & e^3 \\
\end{array}$$
with boundary maps 
$$\begin{array}{r l}
\partial e^3 = & e_1^2 + e^2_2 + e^2_3 - e_1^2 - e^2_2 - e^2_3 = 0 \\
\partial e_1^2 = & a + b - a - b = 0 \\
\partial e_2^2 = & a + c - f(a) - c \\
\partial e_3^2 = & b + c - f(b) - c \\
\end{array}$$
All other attaching maps are trivial. 

\begin{figure}
\begin{center}
\begin{tikzpicture}
\draw [fill=lightgray,thick] (0,0) rectangle (2,2);
\draw[thick] (0,0) -- node[below] {$a$} ++ (2,0);
\draw[thick, ->] (0,0) -- (1,0);

\draw[thick] (0,0) -- node[left] {$b$} ++ (0,2);
\draw[thick, ->] (0,0) -- (0,1);

\draw[thick] (0,2) -- node[above] {$a$} ++ (2,0);
\draw[thick, ->] (0,2) to (1,2);

\draw[thick] (2,0) -- node[right] {$b$} ++ (0,2);
\draw[thick, ->] (2,0) to (2,1);

\draw[thick, ->] (2,0) to [out=90,in=90] (4,0);
\draw[thick] (2,0) to [out=-90,in=-90] (4,0);
 \node [label={$c$},shift={(3,0)}]{};
\end{tikzpicture}
\end{center}
\caption{The $1$-skeleton of our mapping torus.}
\end{figure}
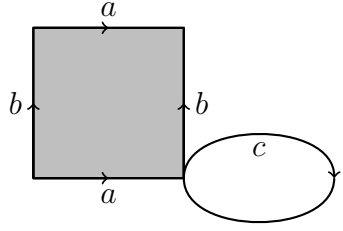

Now let $\tau : T^2 \rightarrow T^2$ denote the Dehn twist about $a$. Note $\tau_*(a) = a$ and $\tau_*(b) = ab$. From the above, we have

\newpage

$$ 
\begin{array}{r l}
\pi_1(N_a) &=  <a,b,c \mid [a,b \,], ac\tau(a)^{-1}c^{-1},  bc \tau(b)^{-1} c^{-1}> \\
                &=  <a,b,c \mid [a,b \,], [a,c \,],  bc b^{-1}a^{-1} c^{-1} >.   
\end{array}
$$
and the following CW-decomposition.
$$\begin{array}{l l}
\text{1 0-cell} & e^0 \\
\text{3 1-cells} & a, b, c \\
\text{3 2-cells} & e_1^2, e^2_2, e^2_3 \\
\text{1 3-cell} & e^3 \\
\end{array}$$
with boundary maps 
$$\begin{array}{r l}
\partial e^3 = & e_1^2 + e^2_2 + e^2_3 - e_1^2 - e^2_2 - e^2_3 = 0 \\
\partial e_1^2 = & a + b - a - b = 0 \\
\partial e_2^2 = & a + c - \tau(a) - c = a +c -a -c = 0 \\
\partial e_3^2 = & b + c - \tau(b) - c = b + c - a - b -c = -a \\
\end{array}$$
All other attaching maps are trivial. 

From this, we calculate the homology to be $H_0(N_a) = \mathbb Z$, $H_1(N_a) = \mathbb Z^2$, $H_2(N_a) = \mathbb Z^2$, and $H_3(N_a) = \mathbb Z$.

\end{section}
\begin{section} {Monodromy trivial in one direction, Dehn twist in the other}

Our desired fibre bundle is given by 
$$
\begin{tikzpicture}[description/.style={fill=white,inner sep=2pt}]
\matrix (m) [matrix of math nodes, row sep=3em,
column sep=2.5em, text height=1.5ex, text depth=0.25ex]
{ S^1 \times S^1     & N_a \times S^1   \\
                     & S^1 \times S^1        \\
};
\path[right hook->,font=\scriptsize]
(m-1-1) edge node[auto]{} (m-1-2);
\path[->,font=\scriptsize]
(m-1-2) edge node[auto]{} (m-2-2);
\end{tikzpicture}
$$
whose projection map is the projection map of $N_a$ on the first factor and the identity on the second factor.
\begin{subsection}{The homology of $\bm{N_a \times S^1}$}

If we endow $S^1$ with the usual CW-structure having one 0-cell $f^0$ and one 1-cell $f^1$, $S^1 \times N_a$ has a CW-structure consisting of:
$$\begin{array}{l l}
\text{1 0-cell} & f^0 \times e^0 \\
\text{4 1-cells} & f^0 \times a, f^0 \times b, f^0 \times c, f^1 \times e^0 \\
\text{3 2-cells} & f^0 \times e_1^2, f^0 \times e^2_2, f^0 \times e^2_3, f^1 \times a, f^1 \times b, f^1 \times c  \\
\text{4 3-cells} & f^0 \times e^3, f^1 \times e_1^2, f^1 \times e^2_2, f^1 \times e^2_3 \\
\text{1 4-cell}  & f^1 \times e^3 \\
\end{array}$$

From this, we calculate the homology to be $H_0(S^1 \times N_a) = \mathbb Z$, $H_1(S^1 \times N_a) = \mathbb Z^3$, $H_2(S^1 \times N_a) = \mathbb Z^4$, and $H_3(S^1 \times N_a) = \mathbb Z^3$, and $H_4(S^1 \times N_a) = \mathbb Z$. 

\end{subsection}

\begin{subsection} {A symplectic form for $\bm{N_a \times S^1}$}

We will construct a symplectic form on $N_a \times S^1$. First consider $\pi : N_a \rightarrow S^1$. Let $\alpha = \pi^*d \theta$ where $d \theta$ is the usual volume form for $S^1$. This will be a closed $1$-form. Now we will show it is possible to choose a Riemannian metric $g$ on $N_a$ that makes $\alpha$ harmonic. We will use the following theorem. Before we state it, we need the following definitions.

We call a $1$-form $\alpha$ {\it intrinsically harmonic} if there exists a metric that makes $\alpha$ harmonic and {\it transitive}, if for each point $p$ not contained in $\alpha$'s zero set, there is a $1$-submanifold containing $p$ on which $\alpha$ restricts to a volume form.

\begin{theorem} (Calabi) \cite{Calabi} A closed $1$-form $\alpha$ having only Morse-type zeros on a closed oriented manifold $M$ is intrinsically harmonic if and only if it is transitive. \hfill $\Box$
\end{theorem}

Choose $p \in N_a$. $\alpha$ has no zeros since $\text{ker } \alpha_p = T_p F \neq T_p N_a$ where $F$ is the fibre containing $p$. To see $\alpha$ is transitive it is sufficient to find a section of $N_a$ containing $p$ since $\alpha$ will restrict to a volume form on each section. Write $p = [x,t_x]$ and choose a path $\gamma : I \rightarrow T^2$ with $\gamma (1) = \tau \circ \gamma(0)$ and $\gamma(t_x) = x$. Then $\sigma = \{ [\gamma(t), t] \mid t \in I \}$ will be a section containing $p$. Therefore by Calabi's theorem, there exists a metric $g$ on $N_a$ that makes $\alpha$ harmonic. Recall since $\alpha$ is intrinsically harmonic, it is a closed form.
 
Let $\omega_1 = \star \alpha \in \Omega^2(N_a)$. This is also closed since since the Hodge operator sends harmonic forms to harmonic forms. Locally, if we choose an oriented orthonormal frame $\{e_1, e_2, e_3 \}$ for $T_p N_a = T_p F \oplus V_p F $ with $\text{span}(e_1, e_2) = T_p F$ and $\text{span}(e_3) = V_p F$. Additionally assume $e_1$ is in the direction of $a$ and $e_2$ is in the direction of $b$. Then we can write $\alpha = c_1 e^3$ since $\text{ker}(\alpha) = T_p F$ and $\omega_1 = \star \alpha = c_1 e^1 \wedge e^2$ since $(3,1,2)$ is an even permutation.

$\omega_1$ induces a $2$-form, $\omega_1 \in \Omega^2(N_a \times S^1)$. Define $\omega = \omega_1 + \omega_2$ where $\omega_2 = \pi^*(\beta \, \, )$, here $\beta \, \,$ is the usual volume form for $T^2$. Locally, if we extend our basis to a basis for $T_p (N_a \times S^1)$, we can write $\omega = c_1 e^1 \wedge e^2 + c_2 e^3 \wedge e^4$. $\omega \wedge \omega = c_1 c_2 e^1 \wedge e^2 \wedge e^3 \wedge e^4$ which is a volume form so $\omega$ is a symplectic form for $N_a \times S^1$.

\end{subsection}

\begin{subsection} {A homology basis for $\bm{H_2(N_a \times S^1)}$}

Now we will find a homology basis for $H_2(N_a \times S^1)$ consisting of surfaces that are symplectic and Lagrangian with respect to the symplectic form we just constructed.

First consider a fibre of $N_a \times S^1$ which we can write $B_1 = F \times \{ pt \}$ where $F$ is a fibre of $N_a$. Locally $T_p B_1$ is spanned by $e_1$ and $e_2$. Let $v = \lambda_1 e_1 + \lambda_2 e_2$ and $w =  \mu_1 e_1 + \mu_2 e_2$. Then
$$\begin{array}{ r l} 
\omega (v, w) = &  \omega_1 (v, w) + \omega_2 (v, w) \\
              = &  \omega_1 (v, w) + e^3 \wedge e^4 (v, w) \\
              = &  \omega_1(v, w) + e^3(v)e^4(w) - e^3(w)e^4(v) \\
              = &  \omega_1(v, w) + 0 - 0  = \omega_1(v, w) \\
\end{array}$$
so $\omega \mid_{B_1} = \omega_1 \mid_{B_1}$ so $B_1$ is a symplectic submanifold. 

Now consider the submanifold $B_2 = c \times S^1$. This is a section of $N_a \times S^1$. Locally $T_p B_2$ is spanned by $e_3$ and $e_4$. Let $v = \lambda_1 e_3 + \lambda_2 e_4$ and $w =  \mu_1 e_3 + \mu_2 e_4$. Then
$$\begin{array}{ r l} 
\omega (v, w) = &  \omega_1 (v, w) + \omega_2 (v, w) \\
              = &  e^1 \wedge e^2 (v, w) + \omega_2 (v, w) \\
              = &  e^1(v)e^2(w) - e^1(w)e^2(v) + \omega_2 (v, w) \\
              = &  0 - 0 + \omega_2 (v, w) = \omega_2 (v, w) \\
\end{array}$$
so $\omega \mid_{B_2} = \omega_2 \mid_{B_2}$ so $B_2$ is also a symplectic submanifold.

Now consider $B_3 = b \times S^1$. Locally $T_p B_3$ is spanned by $e_2$ and $e_4$. Let $v = \lambda_1 e_2 + \lambda_2 e_4$ and $w =  \mu_1 e_2 + \mu_2 e_4$. Then
$$\begin{array}{ r l} 
\omega (v, w) = &  \omega_1 (v, w) + \omega_2 (v, w) \\
              = &  e^1 \wedge e^2 (v, w) + \omega_2 (v, w) \\
              = &  e^1(v)e^2(w) - e^1(w)e^2(v) + e^3(v)e^4(w) - e^3(w)e^4(v) \\
              = &  0e^2(w) - 0e^2(v) + 0e^4(w) - 0e^4(v) = 0 \\
\end{array}$$
so $\omega \mid_{B_3} = 0$ so $B_3$ is a Lagrangian submanifold.

Finally let  $B_4 = [a \times I] \times \{ pt \}$. Locally $T_p B_4$ is spanned by $e_1$ and $e_3$. Let $v = \lambda_1 e_1 + \lambda_2 e_3$ and $w =  \mu_1 e_1 + \mu_2 e_3$. Then
$$\begin{array}{ r l} 
\omega (v, w) = &  \omega_1 (v, w) + \omega_2 (v, w) \\
              = &  e^1 \wedge e^2 (v, w) + \omega_2 (v, w) \\
              = &  e^1(v)e^2(w) - e^1(w)e^2(v) + e^3(v)e^4(w) - e^3(w)e^4(v) \\
              = &  e^1(v)0 - e^1(w)0 + e^3(v)0 - e^3(w)0 = 0 \\
\end{array}$$
so $\omega \mid_{B_4} = 0$ so $B_4$ is also a Lagrangian submanifold.

The fundamental classes of these submanifolds in our CW decomposition are $[f^0 \times e_2^1]$, $[f^1 \times c]$, $[f^1 \times b]$, and $[f^0 \times e_2^2]$ which generate $H_2(N_a \times S^1)$.

\end{subsection}

\end{section}

\begin{section} {Monodromy Dehn twists in both directions}
This is given by the fibre bundle,
$$
\begin{tikzpicture}[description/.style={fill=white,inner sep=2pt}]
\matrix (m) [matrix of math nodes, row sep=3em,
column sep=2.5em, text height=1.5ex, text depth=0.25ex]
{ \sum_2     & N_a \times S^1 \#_{T^2} S^1 \times N_a = M  \\
                     & S^1 \times S^1        \\
};
\path[->,font=\scriptsize]
(m-1-1) edge node[auto]{} (m-1-2);
\path[->,font=\scriptsize]
(m-1-2) edge node[auto]{} (m-2-2);
\end{tikzpicture}
$$
where on the left $T^2$ is the submanifold $[\{ pt \} \times I] \times S^1 \subset N_a \times S^1$. Here the $\{ pt \}$ is chosen so it misses the curve $a$ hence $[I \times \{ pt \}] \cong S^1$ in $N_a$. Similarly on the right, we will cut out an $S^1 \times [\{ pt \} \times I]$. Note $v(T^2) \cong D^2 \times T^2$. 

More specifically, we will embed $T^2$ into $N_a \times S^1$ as follows: $(z_0, z_1) \mapsto ([\{ pt \}, i(z_0)], z_1)$. Here $i : S^1 \rightarrow I$ is a map where $q \circ i :S^1 \rightarrow I \rightarrow \frac{I}{0 \sim 1} = S^1$ is the identity. Note $i$ is not continuous but the composition $q \circ i$ is smooth. Observe $\pi_1([\{ pt \}, i(z_0)], z_1) = (z_0, z_1)$ where $\pi_1$ is the projection map of $N_a \times S^1$.

Similarly we will embed $T^2$ into $S^1 \times N_a$ via the map $(z_0, z_1) \mapsto (z_0, [\{ pt \}, i(z_1)])$. Also observe $\pi_2(z_0, [\{ pt \}, i(z_1)]) = (z_0, z_1)$ where $\pi_2$ is the projection map of $S^1 \times N_a$. 

Now we will define a diffeomorphism 
$$\begin{array}{r l}
\phi : T^2 \subset N_a \times S^1  & \rightarrow T^2 \subset S^1 \times N_a \\
([\{ pt \}, i(z_0)], z_1) & \mapsto (z_0, [\{ pt \}, i(z_0)]) \\ 
\end{array}
$$ 
Note from earlier, this preserves fibres so the above construction is indeed a fibre bundle.

If we lift $\phi$ to an orientation-reversing diffeomorphism $\Phi : \upsilon T^2 \rightarrow \upsilon T^2$, our manifold becomes a symplectic normal connected sum.

\begin{subsection} {The homology of $\bm{M}$}

Now we will use Mayer-Vietoris to calculate the homology of $M$. First we will define a CW structure for $N_a - \upsilon S^1 = N_a - D^2 \times S^1$ where $D^2 \times S^1 = [D^2 \times I]$. Here the $D^2 \subset T^2$ is chosen so it misses $a$. We have 
$$\begin{array}{l l}
\text{1 0-cell} & e^0 \\
\text{4 1-cells} & a, b, c, d \\
\text{4 2-cells} & e_1^2, e^2_2, e^2_3, e^2_4 \\
\text{1 3-cell} & e^3 \\
\end{array}$$
with boundary maps 
$$\begin{array}{r l}
\partial e^3 = & e_1^2 + e^2_2 + e^2_3 - e_1^2 - e^2_2 - e^2_3 + e^2_4 = e^2_4 \\
\partial e_1^2 = & a + b - a - b + d = d \\
\partial e_2^2 = & a + c - \tau(a) - c = a +c -a -c = 0 \\
\partial e_3^2 = & b + c - \tau(b) - c = b + c - a - b -c = -a \\
\partial e_4^2 = & c + d - c - d = 0 \\
\end{array}$$
All other attaching maps are trivial. From this, we calculate the homology to be $H_0(N_a - \upsilon S^1) = \mathbb Z$, $H_1(N_a - \upsilon S^1) = \mathbb Z^2$, and $H_2(N_a - \upsilon S^1) = \mathbb Z$.

Now for $N_a \times S^1 - \upsilon T^2 \cong (N_a - S^1 \times D^2) \times S^1$. If we endow $S^1$ with the usual CW-structure having one 0-cell $f^0$ and one 1-cell $f^1$, we have a CW-structure consisting of:
$$\begin{array}{l l}
\text{1 0-cell} & f^0 \times e^0 \\
\text{5 1-cells} & f^0 \times a, f^0 \times b, f^0 \times c, f^0 \times d, f^1 \times e^0 \\
\text{8 2-cells} & f^0 \times e_1^2, f^0 \times e^2_2, f^0 \times e^2_3, f^0 \times e^2_4, f^1 \times a, f^1 \times b, f^1 \times c, f^1 \times d   \\
\text{5 3-cells} & f^0 \times e^3, f^1 \times e_1^2, f^1 \times e^2_2, f^1 \times e^2_3, f^1 \times e^2_4 \\
\text{1 4-cell}  & f^1 \times e^3 \\
\end{array}$$
From this, we calculate the homology to be $H_0(N_a \times S^1 - \upsilon T^2) = \mathbb Z$, $H_1(N_a \times S^1 - \upsilon T^2) = \mathbb Z^3$, $H_2(N_a \times S^1 - \upsilon T^2) = \mathbb Z^3$, and $H_3(N_a \times S^1 - \upsilon T^2) = \mathbb Z$.

Note $v (T^2) \backsimeq T^3$ has a CW structure inherited from our CW-structure for $N_a \times S^1 - \upsilon T^2$ consisting of
$$\begin{array}{l l}
\text{1 0-cell} & f^0 \times e^0 \\
\text{3 1-cells} & f^0 \times c, f^0 \times d, f^1 \times e^0 \\
\text{3 2-cells} & f^0 \times e_4^2, f^1 \times c, f^1 \times d   \\
\text{1 3-cells} & f^1 \times e^2_4 \\
\end{array}$$
From this, we see the inclusion map $i : T^3 \hookrightarrow N_a \times S^1 - D^2 \times T^2$ has $i_* H_0(T^3) = \mathbb Z$, $i_* H_1(T^3) = \mathbb Z^2$, $i_* H_2(T^3) = \mathbb Z$, and $i_* H_3(T^3) = 0$. 

Let $A$ denote and $B$ denote $N_a \times S^1 - \upsilon T^2$ and $S^1 \times N_a - \upsilon T^2$. Now finally writing out the Mayer-Vietoris sequence:
$$
\begin{tikzpicture}[description/.style={fill=white,inner sep=2pt}]
\matrix (m) [matrix of math nodes, row sep=1em,
column sep=1em, text height=1.5ex, text depth=0.25ex]
{ 
        0             &        &   \mathbb Z                   & \,  \\       
H_4(A) \oplus H_4(B) & H_4(M) & H_3(A \cap B)        & \, \\
        \mathbb Z^2             &        &  \mathbb Z^3                     & \, \\
H_3(A) \oplus H_3(B) & H_3(M) & H_2(A \cap B)        & \, \\
        \mathbb Z^6             &        & \mathbb Z^3                     & \,\\
H_2(A) \oplus H_2(B) & H_2(M) & H_1(A \cap B)        &  \,\\
        \mathbb Z^6             &        &  \mathbb Z                    &  \,\\
H_1(A) \oplus H_1(B) & H_1(M) & H_0(A \cap B)        &  \, \\
        \mathbb Z^2            &        &                      &  \,\\
H_0(A) \oplus H_0(B) & H_0(M) & 0        &  \, \\
};
\path[->,font=\scriptsize]
(m-1-1) edge node[auto]{$\cong$} (m-2-1);
\path[->,font=\scriptsize]
(m-1-3) edge node[auto]{$\cong$} (m-2-3);

\path[->,font=\scriptsize]
(m-3-1) edge node[auto]{$\cong$} (m-4-1);
\path[->,font=\scriptsize]
(m-3-3) edge node[auto]{$\cong$} (m-4-3);

\path[->,font=\scriptsize]
(m-5-1) edge node[auto]{$\cong$} (m-6-1);
\path[->,font=\scriptsize]
(m-5-3) edge node[auto]{$\cong$} (m-6-3);

\path[->,font=\scriptsize]
(m-7-1) edge node[auto]{$\cong$} (m-8-1);
\path[->,font=\scriptsize]
(m-7-3) edge node[auto]{$\cong$} (m-8-3);

\path[->,font=\scriptsize]
(m-9-1) edge node[auto]{$\cong$} (m-10-1);

\path[->,font=\scriptsize]
(m-2-1) edge node[auto]{$\psi_4$} (m-2-2);
\path[->,font=\scriptsize]
(m-2-2) edge node[auto]{$\partial_4$} (m-2-3);
\path[->,font=\scriptsize]
(m-2-3) edge node[auto]{$\Phi_3$} (m-2-4);

\path[->,font=\scriptsize]
(m-4-1) edge node[auto]{$\psi_3$} (m-4-2);
\path[->,font=\scriptsize]
(m-4-2) edge node[auto]{$\partial_3$} (m-4-3);
\path[->,font=\scriptsize]
(m-4-3) edge node[auto]{$\Phi_2$} (m-4-4);

\path[->,font=\scriptsize]
(m-6-1) edge node[auto]{$\psi_2$} (m-6-2);
\path[->,font=\scriptsize]
(m-6-2) edge node[auto]{$\partial_2$} (m-6-3);
\path[->,font=\scriptsize]
(m-6-3) edge node[auto]{$\Phi_1$} (m-6-4);

\path[->,font=\scriptsize]
(m-8-1) edge node[auto]{$\psi_1$} (m-8-2);
\path[->,font=\scriptsize]
(m-8-2) edge node[auto]{$\partial_1$} (m-8-3);
\path[->,font=\scriptsize]
(m-8-3) edge node[auto]{$\Phi_0$} (m-8-4);

\path[->,font=\scriptsize]
(m-10-1) edge node[auto]{$\psi_0$} (m-10-2);
\path[->,font=\scriptsize]
(m-10-2) edge node[auto]{} (m-10-3);
\end{tikzpicture}
$$

First $H_4(M) = \mathbb Z$. From our previous work, we know $\text{Im } \Phi_3 = 0$ and $\text{ker } \Phi_2 = \mathbb Z^2$ which implies $H_3(M) = \mathbb Z^{2+2} = \mathbb Z^4$. $\text{Im } \Phi_2 = \mathbb Z$ and $\text{ker } \Phi_1 = \mathbb Z$ implies $H_2(M) = \mathbb Z^{(6-1)+1} = \mathbb Z^6$. $\text{Im } \Phi_1 = \mathbb Z^2$ and $\text{ker } \Phi_0 = 0$ implies $H_1(M) = \mathbb Z^{(6-2)+0} = \mathbb Z^4$. Finally $\text{Im } \Phi_0 = \mathbb Z$ implies $H_0(M) = \mathbb Z$.

\end{subsection}

\begin{subsection}{A homology basis for $\bm{H_2(M)}$}

Now we will find a basis for $H^2(M)$ consisting of symplectic and Lagrangian submanifolds. Let $i : N_a \times S^1 - \upsilon T^2  \hookrightarrow M$ and $j : S^1 \times N_a - \upsilon T^2 \hookrightarrow M$. By switching the order of the summands, we will also denote the corresponding surfaces in $S^1 \times N_a - \upsilon T^2$ with $B_i$. Note only the fibre $B_1$ doesn't miss $\upsilon T^2 $. From the above Mayer-Vietoris sequence, observe $\text{ker} \psi_2 = < [c \times S^1], [S^1 \times c]> = <[B_2], [B_2]>$ so $[i(B_2)]$, $[i(B_3)]$, $[i(B_4)]$, $[j(B_3)]$, $[j(B_4)]$ are all nonzero and linearly independent in $H_2(M)$. To show these are symplectic and Lagrangian, we'll use the following theorem.

\begin{theorem} (Gompf) \cite{constr} Let $M_1$ and $M_2$ be closed symplectic $4$-manifolds and $\psi_1 : N \hookrightarrow M_1$ and $\psi_2 : N \hookrightarrow M_2$ be symplectic embeddings of a closed connected symplectic $2$-manifold $N$. Then $M_1 \#_N M_2$ admits a symplectic structure. Moreover this structure can be chosen in such a way that the symplectic (Lagrangian) submanifolds of $M_i$ missing $\upsilon\psi_i(N)$ are symplectic (Lagrangian) in $M_1 \#_N M_2$. \hfill $\Box$
\end{theorem}

So with this choice of symplectic structure, we see $i(B_2)$ is symplectic and $i(B_3)$, $i(B_4)$, $j(B_3)$, $j(B_4)$ will all be Lagrangian.

The final element of our basis will be the fibre of $M$, which we will denote by $D$. To see $[D]$ isn't contained in $\text{img } \psi_2$ observe if we write $D = D_1 \cup D_2$ where $D_1 = D \cap A$ and $D_2 = D \cap B$, $\partial D_1 = d \times \{ pt \}$. Also this shows $[D] \neq 0$. To see $D$ is symplectic we'll use the following corollary.

\begin{corollary} \cite{constr}
If in addition, $D_i \subset M_i$ are closed symplectically embedded surfaces that intersect $\psi_1(N)$ and $\psi_2(N)$ transversally in $\ell$ points that have positive sign ($\ell$ independent of $i$) then we may assume $D_1 \#_L D_2$ is a symplectic submanifold of $M_1 \#_N M_2$ where $L$ is the set of $\ell$ points. \hfill $\Box$
\end{corollary}

Thus since the fibre of $S^1 \times N_a$ is symplectic, $D$ is symplectic in $M$.

We can use this basis to show that $M$ is spin. Recall the adjunction formula for symplectic $4$-manifolds.

\begin{theorem}(Adjunction formula) \cite{mcduff} Suppose $(M, \omega)$ is a symplectic $4$-manifold. If $\Sigma \subset M$ is a symplectic surface then
$$K \cdot \Sigma + \Sigma \cdot \Sigma = 2g(\Sigma) - 2$$
where $K \in H_2(M)$ is the canonical class of $(M, \omega)$. \hfill $\Box$
\end{theorem}   

For the fibre $D$, we have $K \cdot D + D \cdot D = 2$. We can perturb $D$ horizontally to obtain $D \cdot D = 0$ hence $K \cdot D = 2$. Similarly for $i(B_2)$ if we perturb vertically, we obtain $i(B_2) \cdot i(B_2) = 0$. Thus $K \cdot i(B_2) = 0$. 

\begin{lemma} \cite{4ManifoldsKirby} \label{lagrange} Given a Lagrangian surface $\Sigma$ of a symplectic $4$-manifold $M$, we have $\Sigma \cdot \Sigma = 2g(\Sigma) - 2$. \hfill $\Box$
\end{lemma}

All the other surfaces in our basis are Lagrangian tori so by Lemma \ref{lagrange} for each of these $\Sigma$, we have $\Sigma \cdot \Sigma = 0$.

From the above and the fact $K$ is a characteristic element, we have 
$$\alpha \cdot \alpha \equiv K \cdot \alpha \text{ (mod $2$)} \equiv 0 \text{ (mod $2$)}$$
 for each $\alpha \in H_2(M)$. Therefore $M$ has an even intersection form and using the following theorem, we see $M$ is indeed spin.

\begin{theorem} \cite{Wild}
A closed orientable $4$-manifold $M$ with no $2$-torsion in $H_1(M)$ and an even intersection form is spin and conversely a spin manifold must have an even intersection form. \hfill $\Box$
\end{theorem}

\end{subsection}

\begin{subsection}{The fundamental group of $\bm{M}$} Now we will use van Kampen to find $\pi_1(M)$. First from our CW structure, we obtain 
$$\pi_1(A) = \pi_1(B)  = <a,b,c,d,f | [a,b \,] d, [a,c \, ],bca^{-1}b^{-1}c^{-1},[c, d \, ]>$$

Now using van-Kampen
thus we can write 
$$\begin{array}{r l l}
\pi_1(M) &= <a_1,b_1,c_1,d_1,f_1,a_2,b_2,c_2,d_2,f_2 |& [a_1,b_1 \,]d_1, [a_1,c_1 \,],b_1c_1a^{-1}_1b^{-1}_1c^{-1}_1,[c_1,d_1 \,], \\
         &                        & [a_2,b_2 \,]d_2, [a_2,c_2 \,],b_2c_2a^{-1}_2b^{-1}_2c^{-1}_2,[c_2,d_2 \,],  \\
         &                        & f_1 c_2^{-1}, c_1 f_2^{-1}, d_1 d_2^{-1}> \\
         &= <a_1,b_1,c_1,d_1,a_2,b_2,c_2,d_2 |& [a_1,b_1 \,]d_1, [a_1,c_1 \,],b_1c_1a^{-1}_1b^{-1}_1c^{-1}_1,[c_1,d_1 \,], \\
         &                        & [a_2,b_2 \,]d_2, [a_2,c_2 \,],b_2c_2a^{-1}_2b^{-1}_2c^{-1}_2,[c_2,d_2 \,],  \\
         &                        & d_1 d_2^{-1}> \\
\end{array}$$    

One of the many ways we can see $M$ isn't $T^4$ is to observe if $\pi_1(M) \cong \mathbb Z^4$ then it is abelian and hence $d_1$ and $d_2$ are trivial. But then we can rewrite
$$\begin{array}{r l}
\pi_1(M) = & <a_1,b_1,c_1 | [a_1,b_1 \,], [a_1,c_1 \,],b_1c_1a^{-1}_1b^{-1}_1c^{-1}_1 > \ast \\
           & <a_2,b_2,c_2 | [a_2,b_2 \,], [a_2,c_2 \,],b_2c_2a^{-1}_2b^{-1}_2c^{-1}_2 >. \\
\end{array}
$$

For this to abelian, we must have $\pi_1(M) = 0$ but this is a contradiction.

\end{subsection}

\end{section}

\begin{section}{Monodromy trivial in one direction, two disjoint Dehn twists on a genus 2 surface in the other}
Consider the fibre bundle,
$$
\begin{tikzpicture}[description/.style={fill=white,inner sep=2pt}]
\matrix (m) [matrix of math nodes, row sep=3em,
column sep=2.5em, text height=1.5ex, text depth=0.25ex]
{ \sum_2     & S^1 \times N_a \#_{T^2} S^1 \times N_a = M  \\
                     & S^1 \times S^1        \\
};
\path[->,font=\scriptsize]
(m-1-1) edge node[auto]{} (m-1-2);
\path[->,font=\scriptsize]
(m-1-2) edge node[auto]{} (m-2-2);
\end{tikzpicture}
$$
where on both the left and right, $T^2$ is the submanifold $S^1 \times [\{ pt \} \times I] \times S^1$. Here the $\{ pt \}$ is chosen so it misses the curve $a$ hence $[\{ pt \} \times I] \cong S^1$ in $N_a$. We will identify each $T^2$ with the identity diffeomorphism.

The homology and basis for $H_2(M)$ are the same as in the previous example as the only difference from the previous bundle is the diffeomorphism gluing the two copies of $T^2$. Thus our manifold will be spin. The fundamental group is given as follows.

$$\begin{array}{r l l}
\pi_1(M) &= <a_1,b_1,c_1,d_1,f_1,a_2,b_2,c_2,d_2,f_2 |& [a_1,b_1 \,]d_1, [a_1,c_1 \,],b_1c_1a^{-1}_1b^{-1}_1c^{-1}_1,[c_1,d_1 \,], \\
         &                        & [a_2,b_2 \,]d_2, [a_2,c_2 \,],b_2c_2a^{-1}_2b^{-1}_2c^{-1}_2,[c_2,d_2 \,],  \\
         &                        & f_1 f_2^{-1}, c_1 c_2^{-1}, d_1 d_2^{-1}> \\
\end{array}$$ 

At this time, it is unknown if this is isomorphic to the fundamental group of our previous manifold and also if the two manifolds are diffeomorphic to each other.

\end{section}

\nocite{*}
% \singlespacing
\bibliographystyle{alpha}
 \bibliography{bibfile}

\appendix

\end{document}